\documentclass[a4paper,10pt]{article}

\usepackage[utf8]{inputenc}
\usepackage[T1]{fontenc}
\usepackage[english]{babel}
\usepackage{mathtools,amsmath,amssymb,amsfonts,amsthm,mathrsfs}
\usepackage{bbm}
\usepackage{indentfirst}
\mathtoolsset{showonlyrefs}

\usepackage{graphicx,color}
\usepackage{tikz,pgfplots,pgf}
\usepackage{epsfig}
\usepackage{subcaption}
\usepackage[font=small]{caption}

\newtheorem{theorem}{Theorem}[section]

\newtheorem{lemma}{Lemma}[section]
\newtheorem{proposition}{Proposition}[section]

\theoremstyle{definition}
\newtheorem{remark}{Remark}[section]
\newtheorem{conjecture}{Conjecture}

\numberwithin{equation}{section}

\newcommand\blfootnote[1]{\begingroup\renewcommand\thefootnote{}\footnote{#1}\addtocounter{footnote}{-1}\endgroup}

\begin{document}

\title{
{\bf\Large An indefinite nonlinear problem in population dynamics: high multiplicity of positive solutions\,}\footnote{
Work partially supported by the Grup\-po Na\-zio\-na\-le per l'Anali\-si Ma\-te\-ma\-ti\-ca, la Pro\-ba\-bi\-li\-t\`{a} e le lo\-ro
Appli\-ca\-zio\-ni (GNAMPA) of the Isti\-tu\-to Na\-zio\-na\-le di Al\-ta Ma\-te\-ma\-ti\-ca (INdAM). Progetto di Ricerca 2017:
``Problemi differenziali con peso indefinito: tra metodi topologici e aspetti dinamici''. Guglielmo Feltrin is supported by the Belgian F.R.S.-FNRS - Fonds de la Recherche Scientifique, \textit{Charg\'{e} de recherches} project: ``Quantitative and qualitative properties of positive solutions to indefinite problems arising from population genetic models: topological methods and numerical analysis'',  and partially by the project ``Existence and asymptotic behavior of solutions to systems of semilinear elliptic partial differential equations'' (T.1110.14).}
}
\author{{\bf\large Guglielmo Feltrin}
\vspace{1mm}\\
{\it\small Department of Mathematics, University of Mons}\\
{\it\small place du Parc 20, B-7000 Mons, Belgium}\\
{\it\small e-mail: guglielmo.feltrin@umons.ac.be}\vspace{1mm}\\
\vspace{1mm}\\
{\bf\large Elisa Sovrano}
\vspace{1mm}\\
{\it\small Department of Mathematics, Computer Science and Physics, University of Udine}\\
{\it\small via delle Scienze 206, 33100 Udine, Italy}\\
{\it\small e-mail: sovrano.elisa@spes.uniud.it}\vspace{1mm}}

\date{}

\maketitle

\vspace{5mm}

\begin{abstract}
\noindent Reaction-diffusion equations have several applications in the field of population dynamics and some of them are characterized by the presence of a factor which describes different types of food sources in a heterogeneous habitat. In this context, to study persistence or extinction of populations it is relevant the search of nontrivial steady states.
Our paper focuses on a one-dimensional model given by a parameter-dependent equation of the form $u'' + \bigl{(} \lambda a^{+}(t)-\mu a^{-}(t) \bigr{)}g(u) = 0$, where $g\colon \mathopen{[}0,1\mathclose{]} \to \mathbb{R}$ is a continuous function such that $g(0)=g(1)=0$, $g(s) > 0$ for every $0<s<1$ and $\lim_{s\to0^{+}}g(s)/s=0$, and the weight $a(t)$ has two positive humps separated by a negative one. In this manner, we consider bounded habitats which include two favorable food sources and an unfavorable one. We deal with various boundary conditions, including the Dirichlet and Neumann ones, and we prove the existence of eight positive solutions when $\lambda$ and $\mu$ are positive and sufficiently large.
Throughout the paper, numerical simulations are exploited to discuss the results and to explore some open problems.
\end{abstract}

\section{Introduction}\label{section-1}

Many biological applications describing phenomena of population dispersal involve reaction-diffusion equations (cf.~\cite{CaCo-03,He-81,LG-16,Mu-89}). 
The problem we take into account in this work is motived by the interest in study the effect of the heterogeneity in finite habitats on populations, whose densities are influenced by location and time (cf.~\cite{CaCo-91} and the references therein).
In this context, most common formulations of reaction-diffusion problems, modeling the density $u=u(x,t)$ of a population, lead to semilinear parabolic problems of the form
\begin{equation}\label{eq-i1}
\begin{cases}
\, \dfrac{\partial u}{\partial t}= d\Delta u + w(x) g(u)  &\text{in }  \Omega\times \mathopen{]}0,+\infty\mathclose{[},\\
\, u(x,0)=u_{0}	 &\text{in  }   \partial \Omega,\\
\, \mathfrak{B}u=0	 &\text{on  }   \partial \Omega\times \mathopen{]}0,+\infty\mathclose{[},
\end{cases}
\end{equation}
where $d>0$ is a real parameter, $\Omega\subseteq\mathbb{R}^{N}$ is a bounded domain, $N\geq 1$, $w\colon\Omega\to\mathbb{R}$ is a weight term, $g\colon\mathcal{I}\to\mathbb{R}^{+}:=\mathopen{[}0,+\infty\mathclose{[}$ is a nonlinear function with $\mathcal{I}=\mathopen{[}0,1\mathclose{]}$ or $\mathcal{I}=\mathopen{[}0,+\infty\mathclose{[}$ such that $g(0)=0$ and $\mathfrak{B}$ is the boundary operator, that it could be of Dirichlet type, i.e.~$\mathfrak{B}u=u$, or of Neumann type, i.e.~$\mathfrak{B}u=\partial u/\partial\mathbf{n}$ with $\mathbf{n}$ the outward unit normal vector on $\partial\Omega$, or an alternation of the previous ones through the boundary. In particular, the Dirichlet boundary condition means that the exterior habitat is hostile, instead, the Neumann one means that there exists an inescapable barrier for the population leading to no-flux across the boundary.

The variation of the population density at location $x$ is modeled by the diffusion component $d\Delta u$, with diffusion rate $d$, and by the reaction component $w(x)g(u)$ which varies with the location, due to the spatial heterogeneity reflected by the function $w(x)$ that is assumed to be sign-changing. 
In other words, we look at the so-called case of \emph{indefinite weight problems} (cf.~\cite{HeKa-80}), which involve a weight term $w(x)$ which is positive, zero or negative in different parts of $\Omega$. This way, one could consider a ``food source''  for the population which, in different regions of the habitat $\Omega$, is good (favorable), neutral or worst (unfavorable), respectively. In this context, solutions to such kind of reaction-diffusion problems usually describe the population densities and also, in the particular case of problems arising in population genetics, the distributions of some gene frequency within a population. 

The search of steady state solutions to \eqref{eq-i1} and the analysis of their stability turn out to be crucial in order to address questions about extinction, persistence or coexistence for population. 
In more detail, to investigate effects of the indefinite weight term on the dynamics of \eqref{eq-i1} it is essential to deal with the steady state problem given by 
\begin{equation}\label{eq-i2}
\begin{cases}
\, -\Delta u = \lambda w(x) g(u)  &\text{in } \Omega,\\
\, \mathfrak{B}u=0	 &\text{on  }  \partial \Omega,
\end{cases}
\end{equation}
where $\lambda:=1/d$ is a positive real parameter. Therefore, some problems in population dynamics concern the study of existence or nonexistence as well as uniqueness or multiplicity of  positive (nontrivial) solutions to the indefinite weight problem \eqref{eq-i2}, by varying the parameter $\lambda$.

In the past decades, a great deal of attention has been devoted to indefinite weight problems, like the ones in \eqref{eq-i2}, starting from the pioneering works \cite{BrLi-80, HeKa-80, LG-96, MaMi-73}.
Dealing with several types of boundary conditions and with a wide variety of weight functions $w(x)$ and nonlinearities $g(u)$, the research on positive solutions to \eqref{eq-i2} has grown up at the end of the Eighties
(cf.~\cite{AlTa-93,AmLG-98,BaPoTe-88,BCN-94,BCN-95,BrOs-86,BrHe-90,Se-83}).
Recent literature shows that it is still a very active area of investigation (cf.~\cite{BoGoHa-05,Bo-11,BoFeZa-16prse,BoGa-16,FeSo-NA,FeZa-15jde,GaHaZa-03,GiGo-09,GoLo-00,LoNa-02,LoNaNi-13,LoNiSu-10,NaNiSu-10,ObOm-06,So-17} and the references therein, in order to cover most of the results achieved with different techniques so far). 

Different kinds of populations dynamics can be described depending on the choice of the nonlinear term $g(u)$. For instance, starting from the pioneering works in population genetics concerning the spatial spread of an advantageous gene carried out by Fisher \cite{Fi-37} and by Kolmogorov-Petrovsky-Piskounov \cite{KPP-37}, generalized then by Fleming \cite{Fl-75} and by Henry \cite{He-81}, it has turned out relevant to consider a family of nonlinearities given by a function $g\colon\mathopen{[}0,1\mathclose{]}\to\mathbb{R}^{+}$ such that 
\begin{equation*}
g(0) = g(1) = 0, \qquad g(s) > 0 \quad \text{for } \, 0 < s < 1.
\leqno{(g_{*})}
\end{equation*}
By dealing with nonlinearities $g(u)$ satisfying $(g_{*})$, the dynamics of problem \eqref{eq-i1} ensures   the existence of a trivial steady state and so a trivial positive solution to~\eqref{eq-i2} given by $u\equiv1$.

In the present work, we focus on nonlinearities $g(u)$ characterized by conditions in $(g_{*})$ and deal with the multiplicity of positive solutions to \eqref{eq-i2} avoiding the trivial one. In this framework, there are some fundamental results in the literature that state the existence of at least two positive solutions for both Dirichlet and Neumann problems that are stepping stones for our investigations.

Under the choice of Dirichlet boundary conditions in \eqref{eq-i2}, we recall the following significant result of multiplicity proved in \cite{Ra-73}.

\begin{theorem}[Rabinowitz, 1973/74]\label{th-Rab}
Let $g \colon \mathopen{[}0,1\mathclose{]} \to \mathbb{R}^{+}$ be a locally Lipschitz continuous function satisfying $(g_{*})$ and $\lim_{s\to 0^{+}} g(s)/s = 0$.
Let $w \colon \Omega \to \mathbb{R}$ be a continuous function such that there exists $x_{0}\in\Omega$ with $w(x_{0})>0$. Then, there exists $\lambda^{*}>0$ such that for each $\lambda>\lambda^{*}$ problem \eqref{eq-i2} with Dirichlet boundary conditions has at least two positive solutions.
\end{theorem}

On the other hand, even if for Dirichlet problems no special assumptions on $w(x)$ are needed, the choice of Neumann boundary conditions in \eqref{eq-i2} leads to necessary conditions on the weight term, in order to ensure the existence of positive solutions. Indeed, for Neumann problems a first necessary condition for the existence of positive solutions requires that $w(x)$ is of indefinite sign (it follows by a straightforward integration of the differential equation in \eqref{eq-i2} with Neumann boundary conditions). However, alongside this remark, the existence or nonexistence of positive solutions for the Neumann problem is also influenced by the sign of $\int_{\Omega}w(x)\,dx$ (cf.~\cite{He-81,LoNa-02,NaNiSu-10,Se-83}). Concerning the multiplicity of positive solutions to the Neumann problem~\eqref{eq-i2}, we recall the following result presented in \cite{LoNiSu-10}.
 
\begin{theorem}[Lou, Ni and Su, 2010]\label{th-LNS}
Let $g \colon \mathopen{[}0,1\mathclose{]} \to \mathbb{R}^{+}$ be a function of class $\mathcal{C}^{1}$ satisfying $(g_{*})$, $g'(0) = 0$ and such that there exists $k_{1}\geq1$ with $\lim_{s\to0^{+}}g(s)/s^{k_{1}}>0$.
Let $w \colon \Omega \to \mathbb{R}$ be a sign-changing continuous function such that $\int_{\Omega}w(x)\,dx<0$. Then, there exists $\lambda^{*}>0$ such that for each $\lambda>\lambda^{*}$ problem \eqref{eq-i2} with Neumann boundary conditions has at least two positive solutions.
\end{theorem}

Our goal is to pursue investigations on multiplicity of positive solutions to indefinite weight problems \eqref{eq-i2}, characterizing the indefinite weight term $w(x)$ in terms of positivity and negativity regions, by introducing real parameters which govern the intensity of ``food sources'' within $\Omega$. 
As in \cite{FeSo-NA}, we take advantage of a one-dimensional setting by considering a one-dimensional region $\Omega=\mathopen{]}0,T\mathclose{[}$. From a biological viewpoint, we notice that such kind of environment appears to be very important in the modeling of population's spread in neighborhood of rivers, seashore or narrow valley (cf.~\cite{LLMM-06,Na-78}). Secondly, we describe the ``food sources'' which are favorable, neutral or unfavorable by means of the intervals of positivity, nullity or negativity of a weight function $a\colon\mathopen{[}0,T\mathclose{]}\to\mathbb{R}$. Consequently, as usual in the context of ordinary differential equations, we denote by $t$ the independent (spatial) variable and we deal with an indefinite boundary value problem of the form
\begin{equation}\label{bvp-intro}
\begin{cases}
\, u'' + a(t) g(u) = 0, \\
\, \mathcal{BC}(u,u') = 0_{\mathbb{R}^{2}},
\end{cases}
\end{equation}
where the weight term $a(t)$ is sign-changing and the nonlinearity $g\colon \mathopen{[}0,1\mathclose{]} \to \mathbb{R}^{+}$ is a continuous function 
satisfying $(g_{*})$ and such that
\begin{equation*}
\lim_{s\to 0^{+}} \dfrac{g(s)}{s} = 0.
\leqno{(g_{0})}
\end{equation*}
Concerning the boundary conditions $
\mathcal{BC}(u,u')$, we investigate the following ones:
\begin{equation}\label{BC}
\mathcal{BC}(u,u')\in\Bigl{\{}\bigl{(}u(0),u(T)\bigr{)},\bigl{(}u(0),u'(T)\bigr{)},\bigl{(}u'(0),u(T)\bigr{)},\bigl{(}u'(0),u'(T)\bigr{)}\Bigr{\}},
\end{equation}
namely Dirichlet, Neumann or mixed-type ones.
Our purpose is the study of positive nontrivial solutions to problem \eqref{bvp-intro}, i.e.~solutions $u(t)$ such that $0< u(t)<1$ for all $t\in\mathopen{[}0,T\mathclose{]}$. 

In the same spirit of \cite{BoFeZa-17tams,BoGa-16dcds,FeSo-NA,SoZa-17}, given $\lambda$ and $\mu$ positive real numbers, we consider the following parameter-dependent weight term
\begin{equation}\label{eq-weight}
a(t)=a_{\lambda,\mu}(t):=\lambda a^{+}(t)-\mu a^{-}(t),
\end{equation}
with $a^{+}(t)$ and $a^{-}(t)$ denoting the positive and the negative part of the function $a(t)$, respectively. We notice that investigations involving such kind of indefinite weight terms  have been introduced in literature by the works \cite{Lo-97,Lo-00}.
Finally, our aim is to discuss the relation between the heterogeneity of the habitat, described by nodal properties of $a(t)$, and the number of positive solutions to~\eqref{bvp-intro}. Indeed, by assuming that the weight term $a(t)$ has two positive humps divided by a negative one, our main achievement addresses multiplicity of positive solutions to problem~\eqref{bvp-intro}.
Indeed, we propose a result about high multiplicity of positive solutions, in comparison with the results of the existence of two positive solutions obtained in the same setting both in Theorem~\ref{th-Rab} and in Theorem~\ref{th-LNS}.

\begin{theorem}\label{th-main-intro}
Let $g \colon \mathopen{[}0,1\mathclose{]} \to \mathbb{R}^{+}$ be a locally Lipschitz continuous function satisfying $(g_{*})$ and $(g_{0})$.
Let $a \colon \mathopen{[}0,T\mathclose{]} \to \mathbb{R}$ be a continuous function for which there exist $\sigma,\tau\in\mathopen{]}0,T\mathclose{[}$, with $\sigma < \tau$, such that $a(t)\geq 0$ for every $t\in\mathopen{[}0,\sigma \mathclose{]}\cup\mathopen{[}\tau,T \mathclose{]}$, $a(t)\leq 0$ for every $t\in\mathopen{[}\sigma,\tau \mathclose{]}$ and $a(t_{i})\not=0$ with $i\in\{1,2,3\}$ for some $t_{1}\in\mathopen{[}0,\sigma \mathclose{]}$, $t_{2}\in\mathopen{[}\sigma,\tau \mathclose{]}$ and $t_{3}\in\mathopen{[}\tau,T \mathclose{]}$. 
Then, there exists $\lambda^{*}>0$ such that for each $\lambda>\lambda^{*}$ there exists $\mu^{*}(\lambda)>0$ such that for every $\mu>\mu^{*}(\lambda)$ problem $\eqref{bvp-intro}$ has at least eight positive solutions.
\end{theorem}

In this work, taking advantage of the one-dimensional framework, we give a full description of the dynamics in the phase-plane $(u,u')$ and exploit the shooting method to find positive nontrivial solutions. In more detail, we obtain our result of multiplicity by studying the trajectories of the phase-plane system associated with \eqref{bvp-intro}, in dependence on $\lambda$ and $\mu$.

\medskip

The paper is structured as follows. Section~\ref{section-2} is devoted to our main result (Theorem~\ref{th-main}): first we examine the problem from phase-plane perspective and then we give the proof. Section~\ref{section-3} collects some numerical experiments and bifurcation diagrams that illustrate the abstract formulation and justify a conjecture on the number of solution when a weight with more sign-changes is considered. A final Appendix contains all the technical details for the study of the solutions on the intervals of positivity or negativity of the weight.

\section{Multiplicity of positive solutions}\label{section-2}

In this section we investigate on nontrivial solutions to the boundary value problem
\begin{equation*}
\begin{cases}
\, u'' + \bigl{(}\lambda a^{+}(t)-\mu a^{-}(t)\bigr{)} g(u) = 0, \\
\, \mathcal{BC}(u,u') = 0_{\mathbb{R}^{2}},
\end{cases}
\leqno{(\mathscr{P}_{\lambda,\mu})}
\end{equation*}
where the boundary conditions $\mathcal{BC}(u,u')$ are defined as in \eqref{BC}.
Through the section we implicitly assume that the function $g \colon \mathopen{[}0,1\mathclose{]} \to {\mathbb{R}}^{+}$ is locally Lipschitz continuous satisfying $(g_{*})$ and $(g_{0})$. We extend the function $g(s)$ continuously to $\mathbb{R}$, by setting
\begin{equation*}
g(s)= 0, \quad \text{for } \, s\in\mathopen{]}-\infty,0\mathclose{[}\cup\mathopen{]}1,+\infty\mathclose{[},
\end{equation*}
and we denote this extension still by $g(s)$.

Moreover, we consider a weight $a\in L^{1}(0,T)$ which satisfies the following sign-condition:
\begin{itemize}
\item[$(a_{*})$] \textit{there exist $\sigma,\tau\in\mathopen{]}0,T\mathclose{[}$, with $\sigma < \tau$, such that $a(t)\succ 0$ on $\mathopen{[}0,\sigma \mathclose{]}$, $a(t)\prec 0$ on $\mathopen{[}\sigma,\tau \mathclose{]}$, and $a(t)\succ 0$ on $\mathopen{[}\tau,T \mathclose{]}$.}
\end{itemize}
The symbol $a(t) \succ 0$ means that $a(t)\geq 0$ almost everywhere 
on a given interval with $a\not\equiv 0$ on that interval, and $a(t) \prec 0$ stands for $-a(t) \succ 0$.

In our context, a \textit{solution} $u(t)$ of problem $(\mathscr{P}_{\lambda,\mu})$ is meant in the Carath\'{e}odory's sense and is such that $0\leq u(t)\leq1$ for all $t\in\mathopen{[}0,T\mathclose{]}$. Moreover, we say that a solution is \textit{positive} if $u(t)>0$ for all $t\in\mathopen{]}0,T\mathclose{[}$.

The strategy adopted here to prove multiplicity of positive (nontrivial) solutions to problem $(\mathscr{P}_{\lambda,\mu})$ relies on the \textit{shooting method}. The standard way to exploit this technique is to consider an equivalent formulation of the problem in the phase-plane $(x,y)=(u,u')$. Accordingly, we introduce the system 
\begin{equation*}
\begin{cases}
\, x' = y, \\
\, y' = - \bigl{(} \lambda a^{+}(t) - \mu a^{-}(t) \bigr{)} g(x).
\end{cases}
\leqno{(\mathcal{S}_{\lambda,\mu})}
\end{equation*}
and we denote by
\begin{equation*}
(x(t;t_{0},x_{0},y_{0}), y(t;t_{0}, x_{0},y_{0}))
\end{equation*}
the (unique) solution to $(\mathcal{S}_{\lambda,\mu})$ satisfying the initial conditions 
\begin{equation}\label{eq-ic}
x(t_{0})=x_{0},\quad y(t_{0})=y_{0}.
\end{equation}
We stress that the solution $(x(\cdot;t_{0}, x_{0},y_{0}), y(\cdot;t_{0}, x_{0},y_{0}))$ is globally defined on $\mathopen{[}0,T\mathclose{]}$.

As consequence, for any given $\lambda$ and $\mu$, we can introduce the \textit{Poincar\'{e} map} $\Phi_{t_{0}}^{t_{1}}$ associated with system $(\mathcal{S}_{\lambda,\mu})$ in the interval $\mathopen{[}t_{0},t_{1}\mathclose{]}\subseteq\mathopen{[}0,T\mathclose{]}$, which is the global diffeomorphism of the plane onto itself defined by 
\begin{equation*}
\Phi_{t_{0}}^{t_{1}}\colon \mathbb{R}^{2} \to \mathbb{R}^{2},\quad \Phi_{t_{0}}^{t_{1}}(x_{0},y_{0}) := (x(t_{1};t_{0},x_{0},y_{0}), y(t_{1};t_{0},x_{0},y_{0})).
\end{equation*}
In this manner, a positive solution to $(\mathscr{P}_{\lambda,\mu})$ is determined by a point $P$ in $\mathopen{]}0,1\mathclose{[}\times\mathbb{R}$ for which there exist $P_{0},P_{T}\in\mathbb{R}^{2}$ such that $\Phi_{0}^{\kappa}(P_{0}) = P = \Phi_{T}^{\kappa}(P_{T})$, for some $\kappa\in\mathopen{]}0,T\mathclose{[}$, where
\begin{itemize}
\item either $P_{0}\in\{0\}\times\mathopen{[}0,+\infty\mathclose{[}$, if $u(0)=0$, or $P_{0}\in\mathopen{[}0,1\mathclose{]}\times\{0\}$, if $u'(0)=0$,
\item either $P_{T}\in\{0\}\times\mathopen{]}-\infty,0\mathclose{]}$, if $u(T)=0$, or $P_{T}\in\mathopen{[}0,1\mathclose{]}\times\{0\}$, if $u'(T)=0$.
\end{itemize}
Thus, it is convenient to introduce the following sets
\begin{equation*}
X_{\mathopen{[}0,1\mathclose{]}}:=\mathopen{[}0,1\mathclose{]}\times\{0\}, \quad
Y_{\geq0}:=\{0\}\times\mathopen{[}0,+\infty\mathclose{[},
\quad
Y_{\leq0}:=\{0\}\times\mathopen{]}-\infty,0\mathclose{]}.
\end{equation*}
So that, we look for intersections of two continua, which are obtained by shooting $\mathcal{A}_{0}$ forward in time over $\mathopen{[}0,\kappa\mathclose{]}$ and by shooting $\mathcal{A}_{T}$ backward in time over $\mathopen{[}\kappa,T\mathclose{]}$, where the sets $\mathcal{A}_{0},\mathcal{A}_{T}$ are chosen in $\{X_{\mathopen{[}0,1\mathclose{]}},Y_{\geq0},Y_{\leq0}\}$, depending on the boundary conditions at $0$ and at $T$. This is the crucial point to take in consideration in the proof of our main result (Theorem~\ref{th-main}) that we will present in the second part of this section.

\subsection{Phase-plane analysis}\label{section-2.1}

Our aim is to describe the deformations of the sets $X_{\mathopen{[}0,1\mathclose{]}}$, $Y_{\geq0}$ and $Y_{\leq0}$, through the action of the Poincar\'{e} map associated with system $(\mathcal{S}_{\lambda,\mu})$.

Preliminarily, in the following remarks, we show the presence of trapping regions and prohibited regions in the phase-plane for solutions to system~$(\mathcal{S}_{\lambda,\mu})$.

\begin{remark}\label{rem-2.1}
Let $t_{0}\in\mathopen{[}0,T\mathclose{[}$ and $y_{0}<0$. If $(x(t),y(t))$ is the solution to system~$(\mathcal{S}_{\lambda,\mu})$ with initial conditions $(x(t_{0}),y(t_{0}))=(0,y_{0})$, then 
\begin{equation}\label{eq1-rem2.1}
x(t;t_{0},0,y_{0})<0,\quad y(t;t_{0},0,y_{0})<0, \quad \text{for all } \, t\in \mathopen{]}t_{0},T\mathclose{]}.
\end{equation}
By contradiction, let $t^{*}\in\mathopen{]}t_{0},T\mathclose{]}$ be the first point such that $y(t^{*})=0$. By integrating $x'=y$, we deduce $x(t)<0$ for all $t\in\mathopen{]}t_{0},t^{*}\mathclose{]}$. Then, from the definition of $g(s)$ for $s<0$, we have $0=y(t^{*})=y_{0}<0$, a contradiction.

Analogously, the following facts hold. Let $t_{1}\in\mathopen{[}0,T\mathclose{[}$ and $y_{1}>0$. If $(x(t),y(t))$ is the solution to system~$(\mathcal{S}_{\lambda,\mu})$ with initial conditions $(x(t_{1}),y(t_{1}))=(1,y_{1})$, then 
\begin{equation}\label{eq2-rem2.1}
x(t;t_{1},0,y_{1})>1,\quad y(t;t_{1},0,y_{1})>0, \quad \text{for all } \, t\in \mathopen{]}t_{1},T\mathclose{]}.
\end{equation} 
Let $t_{0}\in\mathopen{]}0,T\mathclose{]}$ and $y_{0}>0$. If $(x(t),y(t))$ is the solution to system~$(\mathcal{S}_{\lambda,\mu})$ with initial conditions $(x(t_{0}),y(t_{0}))=(0,y_{0})$, then 
\begin{equation*}
x(t;t_{0},0,y_{0})<0,\quad y(t;t_{0},0,y_{0})>0, \quad \text{for all } \, t\in \mathopen{[}0,t_{0}\mathclose{[}.
\end{equation*}
Let $t_{1}\in\mathopen{]}0,T\mathclose{]}$ and $y_{1}<0$. If $(x(t),y(t))$ is the solution to system~$(\mathcal{S}_{\lambda,\mu})$ with initial conditions $(x(t_{1}),y(t_{1}))=(1,y_{1})$, then 
\begin{equation*}
x(t;t_{1},0,y_{1})>1,\quad y(t;t_{1},0,y_{1})<0, \quad \text{for all } \, t\in \mathopen{[}0,t_{1}\mathclose{[}.
\end{equation*} 

In this manner, we have pointed out the existence of four trapping regions.
$\hfill\lhd$
\end{remark}

\begin{remark}\label{rem-2.2}
Taking into account our hypotheses on the nonlinear term $g(s)$, we notice that there are some regions of the phase-plane that are never crossed by the image of the vertical strip $\mathopen{[}0,1\mathclose{]}\times\mathbb{R}$ through the Poincar\'{e} map. In more detail, we have that
\begin{equation*}
\Phi_{0}^{\eta}(\mathopen{[}0,1\mathclose{]}\times \mathbb{R})\cap E_{i} = \emptyset, \quad \text{for all $\eta\in\mathopen{[}0,T\mathclose{]}$}, \quad i=0,1,
\end{equation*}
where $E_{0}:=\mathopen{]}-\infty,0\mathclose{[}\times\mathopen{[}0,+\infty\mathclose{[}$ and $E_{1}:=\mathopen{]}1,+\infty\mathclose{[}\times\mathopen{]}-\infty,0\mathclose{]}$.

By contradiction, let us suppose that there exist $\eta\in\mathopen{[}0,T\mathclose{]}$ and initial conditions $(x_{0},y_{0})\in \mathopen{[}0,1\mathclose{]}\times \mathbb{R}$ with
\begin{equation*}
x(\eta;0,x_{0},y_{0})<0, \quad y(\eta;0,x_{0},y_{0})\geq0, \qquad \text{(for $i=0$)},
\end{equation*}
or
\begin{equation*}
x(\eta;0,x_{0},y_{0})>1, \quad y(\eta;0,x_{0},y_{0})\leq0, \qquad \text{(for $i=1$)}.
\end{equation*}
Then, there exists (a first) $t_{*}\in\mathopen{]}0,\eta\mathclose{[}$ such that 
\begin{equation*}
x(t_{*};0,x_{0},y_{0})=0, \quad y(t_{*};0,x_{0},y_{0})<0,
\end{equation*}
or, respectively, there exists (a first) $t^{*}\in\mathopen{]}0,\eta\mathclose{[}$ such that 
\begin{equation*}
x(t^{*};0,x_{0},y_{0})=1, \quad y(t^{*};0,x_{0},y_{0})>0.
\end{equation*}
This is in contradiction with \eqref{eq1-rem2.1} and \eqref{eq2-rem2.1}, respectively.

Analogously, it follows that
\begin{equation*}
\Phi_{T}^{\eta}(\mathopen{[}0,1\mathclose{]}\times \mathbb{R})\cap E'_{i} = \emptyset, \quad \text{for all $\eta\in\mathopen{[}0,T\mathclose{]}$}, \quad i=0,1,
\end{equation*}
where $E'_{0}:=\mathopen{]}-\infty,0\mathclose{[}\times\mathopen{]}-\infty,0\mathclose{]}$ and $E'_{1}:=\mathopen{]}1,+\infty\mathclose{[}\times\mathopen{[}0,+\infty\mathclose{[}$.
$\hfill\lhd$
\end{remark}

Now, let us fix an arbitrary
\begin{equation*}
\kappa\in\mathopen{]}\sigma,\tau\mathclose{[}
\end{equation*}
and study system $(\mathcal{S}_{\lambda,\mu})$ on the interval $\mathcal{[}0,\kappa\mathclose{]}$.
The following two propositions guarantee the existence of three sub-continua of $\Phi_{0}^{\kappa}(Y_{\geq0})$ and of $\Phi_{0}^{\kappa}(X_{\mathopen{[}0,1\mathclose{]}})$, respectively, which connect $\{0\}\times\mathopen{]}-\infty,0\mathclose{]}$ with $\{1\}\times\mathopen{]}0,+\infty\mathclose{[}$.
We stress that they are valid also for $\kappa=\tau$.

\begin{proposition}\label{prop-2.1}
There exists $\lambda^{\sharp}_{1}>0$ such that for each $\lambda>\lambda^{\sharp}_{1}$ the following holds. There exists $\mu^{\sharp}_{1}(\lambda)>0$ such that for every $\mu>\mu^{\sharp}_{1}(\lambda)$ there exist
\begin{equation*}
0=\beta_{0}<\beta_{1}<\beta_{2}<\beta_{3}<\beta_{4}<\beta_{5}
\end{equation*}
such that
\begin{equation*}
\Phi_{0}^{\kappa}(\{0\}\times\mathopen{]}\beta_{i},\beta_{i+1}\mathclose{[})\subseteq\mathopen{]}0,1\mathclose{[}\times\mathbb{R}, \quad \text{for $i=0,2,4$},
\end{equation*}
and
\begin{align*}
&\Phi_{0}^{\kappa}(0,\beta_{i})\in\{0\}\times\mathopen{]}-\infty,0\mathclose{]}, \quad \text{for $i=0,3,4$},\\
&\Phi_{0}^{\kappa}(0,\beta_{i})\in\{1\}\times\mathopen{]}0,+\infty\mathclose{[}, \quad \text{for $i=1,2,5$}.
\end{align*}
\end{proposition}

\begin{proof}
We divide the proof into two steps. In the first one we focus on the interval $\mathopen{[}0,\sigma\mathclose{]}$ and in the second one on $\mathopen{[}\sigma,\kappa\mathclose{]}$.

\smallskip
\noindent
\textit{Step~1. Dynamics on $\mathopen{[}0,\sigma\mathclose{]}$. }
In the same spirit of \cite{GaHaZa-03, GaHaZa-04}, let us consider the set
\begin{equation*}
\mathscr{B} := \bigl{\{} \beta >0 \colon 0<x(t;0,0,\beta)<1, \, \forall \, t\in\mathopen{]}0,\sigma\mathclose{]}\bigr{\}} \subseteq \mathopen{]}0,+\infty\mathclose{[}.
\end{equation*}
By an application of Rabinowitz's Theorem~\ref{th-Rab}  on the interval $\mathopen{[}0,\sigma\mathclose{]}$, there exists $\lambda^{\sharp}_{1}>0$ such that $\mathscr{B}\neq\mathopen{]}0,+\infty\mathclose{[}$ for every $\lambda>\lambda^{\sharp}_{1}$. By Lemma~\ref{lem-A1.1}, taking $\beta>0$ sufficiently small, we obtain $\beta\in\mathscr{B}$ and, moreover, $x(t;0,0,\beta)\in\mathopen{]}0,1\mathclose{[}$ and $y(t;0,0,\beta)>0$, for all $t\in\mathopen{]}0,\sigma\mathclose{]}$.
Let $\beta^{*}:=\sup\{\beta \colon \mathopen{]}0,\beta\mathclose{]} \subseteq\mathscr{B} \}$.
We remark that $x(\sigma;0,0,\beta^{*})=0$. Indeed, if by contradiction $x(\sigma;0,0,\beta^{*})>0$, then by continuity there exists $\beta'>\beta^{*}$ such that $x(\sigma;0,0,\beta')>0$. Hence, from the concavity, we have $x(t;0,0,\beta')>0$ for all $t\in\mathopen{]}0,\sigma\mathclose{]}$, a contradiction with the definition of supremum. Then, we straightforward deduce $y(\sigma;0,0,\beta^{*})<0$. Consequently, we have that $\Phi_{0}^{\sigma}(0,\beta^{*}) \in\{0\}\times\mathopen{]}-\infty,0\mathclose{[}$ and $\Phi_{0}^{\sigma}(\{0\}\times\mathopen{]}0,\beta^{*}\mathclose{[}) \subseteq \mathopen{]}0,1\mathclose{[}\times\mathbb{R}$.

At this point, by Lemma~\ref{lem-A1.2} there exits $\delta^{*}>\beta^{*}$ such that $\Phi^{\sigma}_{0}(0,\delta^{*})\in\mathopen{]}1,+\infty\mathclose{[}\times\mathopen{]}0,+\infty\mathclose{[}$.
Then, recalling that $\Phi_{0}^{\sigma}(0,\beta^{*}) \in\{0\}\times\mathopen{]}-\infty,0\mathclose{[}$, from the continuous dependence of the solutions upon the initial data and the intermediate value theorem, via Remark~\ref{rem-2.2}, the following fact holds. 
There exists an interval $\mathopen{[}\delta_{1},\delta_{2}\mathclose{]}\subseteq\mathopen{[}\beta^{*},\delta^{*}\mathclose{]}$ such that $\Phi^{\sigma}_{0}(\{0\}\times\mathopen{]}\delta_{1},\delta_{2}\mathclose{[})\subseteq\mathopen{]}0,1\mathclose{[}\times\mathbb{R}$, $\Phi_{0}^{\sigma}(\delta_{1},0) \in \{0\}\times\mathopen{]}-\infty,0\mathclose{[}$
and $\Phi_{0}^{\sigma}(\delta_{2},0) \in \{1\}\times\mathopen{]}0,+\infty\mathclose{[}$.

\smallskip
\noindent
\textit{Step~2. Dynamics on $\mathopen{[}\sigma,\kappa\mathclose{]}$. }
We fix $\lambda > \lambda^{\sharp}_{1}$.
First of all, we observe that, for any $y_{0}\in\mathbb{R}^{+}$, the solution $(x(t),y(t))$ to system~$(\mathcal{S}_{\lambda,\mu})$ with initial values $(x(0),y(0))=(0,y_{0})$ satisfies
\begin{equation*}
y(\sigma) = y(0) - \lambda \int_{0}^{\sigma} a^{+}(\xi)g(x(\xi)) \,d\xi \geq -\lambda^{*}_{1} \|a^{+}\|_{L^{1}(0,\sigma)} \max_{s\in\mathopen{[}0,1\mathclose{]}}g(s) =: -\omega_{\sigma}.
\end{equation*}
Let us take $p\in\mathopen{]}0,\beta^{*}\mathclose{[}$ and define $\nu_{\sigma}:=x(\sigma;0,0,p)$.
From the properties of the continua $\Phi_{0}^{\sigma}(\mathopen{[}0,\beta^{*}\mathclose{]}\times\{0\})$ achieved in \textit{Step~1}, it follows that $\nu_{\sigma}\in\mathopen{]}0,1\mathclose{[}$.
Next, by fixing $\nu_{2}\in\mathopen{]}\nu_{\sigma},1\mathclose{[}$, we choose $t_{2}$ such that
\begin{equation*}
\sigma< t_{2} \leq \min\biggl{\{}\sigma + \dfrac{\nu_{\sigma}}{2\omega_{\sigma}},
\dfrac{\sigma(1-\nu_{2})+\kappa(\nu_{2}-\nu_{\sigma})}{1-\nu_{\sigma}} 
\biggr{\}}
\end{equation*}
and $\omega>(1-\nu_{2})/(\kappa-t_{2})$.
We are now in position to apply Lemma~\ref{lem-A3.1} and Lemma~\ref{lem-A3.2}. Hence, there exists $\mu^{\sharp}_{1}(\lambda):=\mu^{\star}(\nu_{2},\nu_{\sigma},t_{2},\omega_{\sigma})>0$ (cf.~\eqref{eq-mu}) such that for $\mu>\mu^{\sharp}_{1}(\lambda)$ we have that 
\begin{equation}\label{eq-p}
x(\kappa;0,0,p)\geq1,\quad y(\kappa;0,0,p)>\omega>0.
\end{equation}
By Remark~\ref{rem-2.1}, we deduce that 
\begin{equation}\label{eq-r1}
\begin{aligned}
&x(\kappa;0,0,\beta^{*})<0, & &y(\kappa;0,0,\beta^{*})<0,
\\
&x(\kappa;0,0,\delta_{1})<0, & &y(\kappa;0,0,\delta_{1})<0,
\\
&x(\kappa;0,0,\delta_{2})>1, & &y(\kappa;0,0,\delta_{2})>0.
\end{aligned}
\end{equation}
Taking into account \eqref{eq-p}, \eqref{eq-r1} and $\Phi_{0}^{\kappa}(0,0)=(0,0)$, thanks to
the continuous dependence of the solutions upon the initial data and the intermediate value theorem, we conclude that there exist three intervals 
\begin{equation*}
\mathopen{[}\beta_{0},\beta_{1}\mathclose{]}\subseteq\mathopen{[}0,p\mathclose{]}, \quad
\mathopen{[}\beta_{2},\beta_{3}\mathclose{]}\subseteq\mathopen{[}p,\beta^{*}\mathclose{]}, \quad
\mathopen{[}\beta_{4},\beta_{5}\mathclose{]}\subseteq\mathopen{[}\delta_{1},\delta_{2}\mathclose{]},
\end{equation*}
such that
\begin{equation*}
\Phi_{0}^{\kappa}(\mathopen{]}\beta_{i},\beta_{i+1}\mathclose{[}\times\{0\})\subseteq\mathopen{]}0,1\mathclose{[}\times\mathbb{R}, \quad \text{for $i\in\{0,2,4\}$},
\end{equation*}
and
\begin{align*}
&\Phi_{0}^{\kappa}(\beta_{i},0)\in\{0\}\times\mathopen{]}-\infty,0\mathclose{]}, \quad \text{for $i\in\{0,3,4\}$},\\
&\Phi_{0}^{\kappa}(\beta_{i},0)\in\{1\}\times\mathopen{[}0,+\infty\mathclose{[}, \quad \text{for $i\in\{1,2,5\}$}.
\end{align*}
Then the thesis follows.
\end{proof}

\begin{proposition}\label{prop-2.2}
There exists $\lambda^{\sharp}_{2}>0$ such that for each $\lambda>\lambda^{\sharp}_{2}$ the following holds.
There exists $\mu^{\sharp}_{2}(\lambda)>0$ such that for every $\mu>\mu^{\sharp}_{2}(\lambda)$ there exist
\begin{equation*}
0=\alpha_{0}<\alpha_{1}<\alpha_{2}<\alpha_{3}<\alpha_{4}<\alpha_{5}<1
\end{equation*}
such that
\begin{equation*}
\Phi_{0}^{\kappa}(\mathopen{]}\alpha_{i},\alpha_{i+1}\mathclose{[}\times\{0\})\subseteq\mathopen{]}0,1\mathclose{[}\times\mathbb{R}, \quad \text{for $i\in\{0,2,4\}$},
\end{equation*}
and
\begin{align*}
&\Phi_{0}^{\kappa}(\alpha_{i},0)\in\{0\}\times\mathopen{]}-\infty,0\mathclose{]}, \quad \text{for $i\in\{0,3,4\}$},\\
&\Phi_{0}^{\kappa}(\alpha_{i},0)\in\{1\}\times\mathopen{]}0,+\infty\mathclose{[}, \quad \text{for $i\in\{1,2,5\}$}.
\end{align*}
\end{proposition}

\begin{proof}
We divide the proof into two steps. In the first one we focus on the interval $\mathopen{[}0,\sigma\mathclose{]}$ and in the second one on $\mathopen{[}\sigma,\kappa\mathclose{]}$.

\smallskip
\noindent
\textit{Step~1. Dynamics on $\mathopen{[}0,\sigma\mathclose{]}$. }
Let us fix $0<\nu_{1}<\nu_{0}<1$ and $0<t_{1} \leq \sigma (1-\nu_{1}/\nu_{0})$. By Lemma~\ref{lem-A1.3} and Lemma~\ref{lem-A1.4}, there exists $\lambda^{\sharp}_{2}:=\lambda^{\star}(\nu_{0},\nu_{1},t_{1})$ (cf.~\eqref{eq-lambda1}) such that, for $\lambda > \lambda^{\sharp}_{2}$, we obtain
\begin{equation*}
x(\sigma;0,\nu_{0},0)\leq0, \quad y(\sigma;0,\nu_{0},0)<0.
\end{equation*}
Next, by the concavity of $x(t)$ in $\mathopen{[}0,\sigma\mathclose{]}$, we notice that $\Phi^{\sigma}_{0}(\mathopen{[}0,1\mathclose{]}\times\{0\})\subseteq\mathopen{]}-\infty,1\mathclose{]}\times\mathopen{]}-\infty,0\mathclose{]}$ and moreover $\Phi_{0}^{\sigma}(1,0)=(1,0)$.
Thus, from the continuous dependence of the solutions upon the initial data and the intermediate value theorem, the following fact holds. 
There exists an interval $\mathopen{[}p_{1},1\mathclose{]}\subseteq\mathopen{[}\nu_{0},1\mathclose{]}$ such that $\Phi^{\sigma}_{0}(\mathopen{]}p_{1},1\mathclose{[}\times\{0\})\subseteq\mathopen{]}0,1\mathclose{[}\times\mathopen{]}-\infty,0\mathclose{[}$ and $\Phi_{0}^{\sigma}(p_{1},0) \in \{0\}\times\mathopen{]}-\infty,0\mathclose{[}$.

Furthermore, by Lemma~\ref{lem-A1.5} there exits $\nu_{2}\in\mathopen{]}0,\nu_{1}\mathclose{[}$ such that $\Phi^{\sigma}_{0}(\mathopen{]}0,\nu_{2}\mathclose{]}\times\{0\})\subseteq\mathopen{]}0,1\mathclose{[}\times\mathopen{]}-\infty,0\mathclose{]}$. Then, recalling that $\Phi_{0}^{\sigma}(\nu_{0},0) \in \mathopen{]}-\infty,0\mathclose{]}\times\mathopen{]}-\infty,0\mathclose{[}$, from the same previous arguments of continuity, there exists an interval $\mathopen{[}0,p_{2}\mathclose{]}\subseteq\mathopen{[}0,\nu_{0}\mathclose{]}$ (with $p_{2}>\nu_{2}$) such that $\Phi^{\sigma}_{0}(\mathopen{]}0,p_{2}\mathclose{[}\times\{0\})\subseteq\mathopen{]}0,1\mathclose{[}\times\mathopen{]}-\infty,0\mathclose{]}$ and $\Phi_{0}^{\sigma}(p_{2},0) \in \{0\}\times\mathopen{]}-\infty,0\mathclose{[}$.

\smallskip
\noindent
\textit{Step~2. Dynamics on $\mathopen{[}\sigma,\kappa\mathclose{]}$. }
Let us fix $\lambda > \lambda^{\sharp}_{2}$.
By same arguments used in the corresponding step of the proof of Proposition~\ref{prop-2.1}, we get the conclusion of the theorem. More precisely, we obtain the existence of $\mu^{\sharp}_{2}(\lambda)>0$ and $0=\alpha_{0}<\alpha_{1}<\alpha_{2}<\alpha_{3}<\alpha_{4}<\alpha_{5}<1$, with
$\mathopen{[}\alpha_{0},\alpha_{1}\mathclose{]} \cup \mathopen{[}\alpha_{2},\alpha_{3}\mathclose{]} \subseteq\mathopen{[}0,p_{2}\mathclose{]}$, $\mathopen{[}\alpha_{4},\alpha_{5}\mathclose{]}\subseteq\mathopen{[}p_{2},1\mathclose{]}$,
satisfying the properties in the statement.
\end{proof}

We conclude this section by stating the propositions concerning the interval
$\mathopen{[}\kappa,T\mathclose{]}$, where the properties of $\Phi_{T}^{\kappa}(Y_{\leq0})$ and $\Phi_{T}^{\kappa}(X_{\mathopen{[}0,1\mathclose{]}})$ are described.
We stress that they are valid also for $\kappa=\sigma$.
The proofs are analogous to the ones of Proposition~\ref{prop-2.1} and Proposition~\ref{prop-2.2}.

\begin{proposition}\label{prop-2.3}
There exists $\lambda^{\sharp}_{3}>0$ such that for each $\lambda>\lambda^{\sharp}_{3}$ the following holds. There exists $\mu^{\sharp}_{3}(\lambda)>0$ such that for every $\mu>\mu^{\sharp}_{3}(\lambda)$ there exist
\begin{equation*}
0=\beta'_{0}>\beta'_{1}>\beta'_{2}>\beta'_{3}>\beta'_{4}>\beta'_{5}
\end{equation*}
such that
\begin{equation*}
\Phi_{T}^{\kappa}(\{0\}\times\mathopen{]}\beta'_{i+1},\beta'_{i}\mathclose{[})\subseteq\mathopen{]}0,1\mathclose{[}\times\mathbb{R}, \quad \text{for $i\in\{0,2,4\}$},
\end{equation*}
and
\begin{align*}
&\Phi_{T}^{\kappa}(0,\beta'_{i})\in\{0\}\times\mathopen{[}0,+\infty\mathclose{[}, \quad \text{for $i\in\{0,3,4\}$},\\
&\Phi_{T}^{\kappa}(0,\beta'_{i})\in\{1\}\times\mathopen{]}-\infty,0\mathclose{[}, \quad \text{for $i\in\{1,2,5\}$}.
\end{align*}
\end{proposition}

\begin{proposition}\label{prop-2.4}
There exists $\lambda^{\sharp}_{4}>0$ such that for each $\lambda>\lambda^{\sharp}_{4}$ the following holds. There exists $\mu^{\sharp}_{4}(\lambda)>0$ such that for every $\mu>\mu^{\sharp}_{4}(\lambda)$ there exist
\begin{equation*}
0=\alpha'_{0}<\alpha'_{1}<\alpha'_{2}<\alpha'_{3}<\alpha'_{4}<\alpha'_{5}<1
\end{equation*}
such that
\begin{equation*}
\Phi_{T}^{\kappa}(\mathopen{]}\alpha'_{i},\alpha'_{i+1}\mathclose{[}\times\{0\})\subseteq\mathopen{]}0,1\mathclose{[}\times\mathbb{R}, \quad \text{for $i\in\{0,2,4\}$},
\end{equation*}
and
\begin{align*}
&\Phi_{T}^{\kappa}(\alpha'_{i},0)\in\{0\}\times\mathopen{[}0,+\infty\mathclose{[}, \quad \text{for $i\in\{0,3,4\}$}\\
&\Phi_{T}^{\kappa}(\alpha'_{i},0)\in\{1\}\times\mathopen{]}-\infty,0\mathclose{[}, \quad \text{for $i\in\{1,2,5\}$}.
\end{align*}
\end{proposition}

\subsection{Main result}\label{section-2.2}

We are now in position to state and prove our main result. We remark that Theorem~\ref{th-main-intro} in the introduction follows as a straightforward corollary.

\begin{theorem}\label{th-main}
Let $g \colon \mathopen{[}0,1\mathclose{]} \to \mathbb{R}^{+}$ be a locally Lipschitz continuous function satisfying $(g_{*})$ and $(g_{0})$.
Let $a \colon \mathopen{[}0,T\mathclose{]} \to \mathbb{R}$ be an $L^{1}$-function satisfying $(a_{*})$. Then, there exists $\lambda^{*}>0$ such that for each $\lambda>\lambda^{*}$ there exists $\mu^{*}(\lambda)>0$ such that for every $\mu>\mu^{*}(\lambda)$ problem $(\mathscr{P}_{\lambda,\mu})$ has at least eight positive solutions.
\end{theorem}

\begin{proof}
We deal separately with the four boundary conditions in $(\mathscr{P}_{\lambda,\mu})$.

\smallskip
\noindent
\textit{Case 1. Dirichlet boundary conditions: $u(0) = u(T) = 0$}.
Let $\kappa\in\mathopen{]}\sigma,\tau\mathclose{[}$. First of all, we notice that any point 
$P\in\Phi_{0}^{\kappa}(Y_{\geq 0})\cap\Phi_{T}^{\kappa}(Y_{\leq0})$ 
determines (univocally) a solution $(x(t;\kappa,P),y(t;\kappa,P))$ of system $(\mathcal{S}_{\lambda,\mu})$ satisfying Dirichlet boundary conditions $x(0;\kappa,P)=x(T;\kappa,P)=0$. Hence, $u(t):=x(t;\kappa,P)$ is a solution of problem $(\mathscr{P}_{\lambda,\mu})$.
Finally, our goal is to find eight distinct points $P_{i}\in\Phi_{0}^{\kappa}(Y_{\geq 0})\cap\Phi_{T}^{\kappa}(Y_{\leq0})$, for $i\in\{1,\ldots,8\}$, belonging to the vertical strip $\mathopen{]}0,1\mathclose{[}\times\mathbb{R}$. 

According to Proposition~\ref{prop-2.1} and Proposition~\ref{prop-2.3}, first we define
\begin{equation*}
\lambda^{*}:=\max \bigl{\{} \lambda^{\sharp}_{1}, \lambda^{\sharp}_{3} \bigr{\}}.
\end{equation*}
Next, we fix $\lambda>\lambda^{*}$ and define 
\begin{equation*}
\mu^{*}(\lambda):=\max \bigl{\{} \mu^{\sharp}_{1}(\lambda),  \mu^{\sharp}_{3}(\lambda) \bigr{\}}.
\end{equation*}
Let $\mu>\mu^{*}(\lambda)$.
The discussion performed in Section~\ref{section-2.1} lead to the following situation.
\begin{itemize}
\item Proposition~\ref{prop-2.1} ensures the existence of three pairwise disjoint sub-continua in $\Phi_{0}^{\kappa}(Y_{\geq 0})$ connecting $\{0\}\times\mathopen{]}-\infty,0\mathclose{]}$ with $\{1\}\times\mathopen{]}0,+\infty\mathclose{[}$.
\item Proposition~\ref{prop-2.3} ensures the existence of three pairwise disjoint sub-continua in $\Phi_{T}^{\kappa}(Y_{\leq0})$ connecting $\{0\}\times\mathopen{[}0,+\infty\mathclose{[}$ with $\{1\}\times\mathopen{]}-\infty,0\mathclose{[}$.
\end{itemize}
From a standard connectivity argument, we deduce the existence of eight distinct intersection points
\begin{equation*}
P_{i}\in\Phi_{0}^{\kappa}(\mathopen{]}\beta_{i},\beta_{i+1}\mathclose{[}\times\{0\})\cap\Phi_{T}^{\kappa}(\mathopen{]}\beta'_{i},\beta'_{i+1}\mathclose{[}\times\{0\}), \quad \text{for $i\in\{0,2,4\}$}.
\end{equation*} 
Moreover, by Remark~\ref{rem-2.1}, for each $i\in\{1,\dots,8\}$, we have that 
\begin{equation*}
\begin{aligned}
&\Phi_{0}^{t}(\xi,0) \in \mathopen{]}0,1\mathclose{[}\times\mathbb{R}, &&\text{for all } \, t\in\mathopen{[}0,\kappa\mathclose{]}, \; \xi\in\mathopen{]}\beta'_{i+1},\beta'_{i}\mathclose{[},
\\&\Phi_{T}^{t}(\xi,0) \in \mathopen{]}0,1\mathclose{[}\times\mathbb{R}, &&\text{for all } \, t\in\mathopen{[}\kappa,T\mathclose{]}, \; \xi\in\mathopen{]}\beta'_{i+1},\beta'_{i}\mathclose{[},
\end{aligned}
\end{equation*}
and so, $P_{i}\in\mathopen{]}0,1\mathclose{[}\times\mathbb{R}$. 
From the above remarks, for $i\in\{1,\ldots,8\}$, we obtain that $u_{i}(t):=x(t;\kappa,P_{i})$ is a solution to  $(\mathscr{P}_{\lambda,\mu})$ with $0 < u_{i}(t) < 1$ on $\mathopen{]}0,T\mathclose{[}$. Then, the thesis follows.

\smallskip
\noindent
\textit{Case 2. Neumann boundary conditions: $u'(0) = u'(T) = 0$}.
In this situation, we take
$\lambda > \lambda^{*}:=\max \bigl{\{} \lambda^{\sharp}_{2}, \lambda^{\sharp}_{4} \bigr{\}}$
and
$\mu > \mu^{*}(\lambda):=\max \bigl{\{} \mu^{\sharp}_{2}(\lambda),  \mu^{\sharp}_{4}(\lambda) \bigr{\}}.$
With these choices, the proof follows exactly the same scheme of \textit{Case 1}, by means of Proposition~\ref{prop-2.2} and Proposition~\ref{prop-2.4}.

\smallskip
\noindent
\textit{Case 3. Mixed boundary conditions, type~1:} $u(0) = u'(T) = 0$.
In this situation, we take
$\lambda > \lambda^{*}:=\max \bigl{\{} \lambda^{\sharp}_{1}, \lambda^{\sharp}_{4} \bigr{\}}$
and
$\mu > \mu^{*}(\lambda):=\max \bigl{\{} \mu^{\sharp}_{1}(\lambda),  \mu^{\sharp}_{4}(\lambda) \bigr{\}}.$
With these choices, the proof follows exactly the same scheme of \textit{Case 1}, by means of Proposition~\ref{prop-2.1} and Proposition~\ref{prop-2.4}.

\smallskip
\noindent
\textit{Case 4. Mixed boundary conditions, type~2:} $u'(0) = u(T) = 0$.
In this situation, we take
$\lambda > \lambda^{*}:=\max \bigl{\{} \lambda^{\sharp}_{2}, \lambda^{\sharp}_{3} \bigr{\}}$
and
$\mu > \mu^{*}(\lambda):=\max \bigl{\{} \mu^{\sharp}_{2}(\lambda),  \mu^{\sharp}_{3}(\lambda) \bigr{\}}).$
With these choices, the proof follows exactly the same scheme of \textit{Case 1}, by means of Proposition~\ref{prop-2.2} and Proposition~\ref{prop-2.3}.

\smallskip

\noindent All boundary conditions listed in \eqref{BC} for problem $(\mathscr{P}_{\lambda,\mu})$ have been considered, hence the proof is completed.
\end{proof}

\section{Discussion: numerical examples and future perspectives}\label{section-3}

In this section we firstly give a graphical description of the solutions to problem $(\mathscr{P}_{\lambda,\mu})$, which helps to understand the dynamics performed in Section~\ref{section-2}. Then, we provide bifurcations diagrams for our model, that bring to light the role played by the parameters $\lambda$ and $\mu$. The discussion ends with some remarks on weight functions $a(t)$ with more than two intervals of positivity.

\medskip

As an example, let us consider the one-dimensional region $\mathopen{[}0,T\mathclose{]}:=\mathopen{[}0,3\mathclose{]}$ and the following functions
\begin{equation}\label{example}
a(t):=\sin(\pi t), \quad t\in\mathopen{[}0,T\mathclose{]}, \quad \text{ and } \quad g(s):=s^{2}(1-s), \quad s\in\mathopen{[}0,1\mathclose{]}.
\end{equation}
We also take $\sigma=1$ and $\tau=2$. In this manner $a(t)$ has two positive humps on $\mathopen{[}0,1\mathclose{]}$ and $\mathopen{[}2,3\mathclose{]}$, separated by a negative one on $\mathopen{[}1,2\mathclose{]}$.
In this framework all the hypotheses for the applicability of Theorem~\ref{th-main} are satisfied.

In order to illustrate the results of the present paper, in this section we focus only on problem $(\mathscr{P}_{\lambda,\mu})$ with Neumann boundary conditions:
\begin{equation}\label{eq-ex}
\begin{cases}
\, u'' + \bigl{(}\lambda \sin^{+}(\pi t)-\mu \sin^{-}(\pi t)\bigr{)} u^{2}(1-u) = 0, \\
\, u'(0)=u'(3).
\end{cases}
\end{equation}
We stress that one could perform analogous considerations dealing with Dirichlet or mixed boundary conditions too in the same framework of \eqref{eq-ex}.

Via a numerical simulation, for $\lambda=20$ and $\mu=500$, we obtain eight positive solutions to the Neumann problem \eqref{eq-ex}, whose graphs are plotted in Figure~\ref{fig-1}. This outcome is in accordance with Theorem~\ref{th-main}. Furthermore, we also refer to Figure~\ref{fig-1} in order to understand the behavior of the solutions $u(t)$ in $\mathopen{[}0,T\mathclose{]}$. 
Firstly, by the sign-condition on $a(t)$, we remark that the solutions are concave in the intervals $\mathopen{[}0,1\mathclose{]}$ and $\mathopen{[}2,3\mathclose{]}$ where $a(t)\succ0$, while they are convex in the interval $\mathopen{[}1,2\mathclose{]}$ where $a(t)\prec0$.
Secondly, we point out some interesting properties of the solutions at $t=0$ and at $t=T$ arising from the discussion in Section~\ref{section-2}.
Indeed, we notice that at $t=0$ there are mainly the following three kinds of behavior: the solution is either ``very small'', ``small'' or ``large''/``near $1$'' (cf.~the graphs of the corresponding solutions in green, blue and red, shown in Figure~\ref{fig-1}). 
Accordingly, in Section~\ref{section-2} we have proved that $u(0)$ can belongs to three pairwise disjoint subintervals $\mathopen{[}\alpha_{i},\alpha_{i+1}\mathclose{]}$ of $X_{\mathopen{[}0,1\mathclose{]}}$, for $i\in\{0,2,4\}$, explaining the meaning of the above classification.
Lastly, we observe that a similar situation holds at $t=T$.

\begin{figure}[htb]
\centering
\begin{tikzpicture}[scale=1]
\begin{axis}[
  tick pos=left,
  tick label style={font=\scriptsize},
          scale only axis,
  enlargelimits=false,
  xtick={0,3},
  ytick={0,1},
  xlabel={\small $t$},
  ylabel={\small $u(t)$},
  max space between ticks=50,
                minor x tick num=2,
                minor y tick num=5,  
every axis x label/.style={
below,
at={(2.3cm,0cm)},
  yshift=-3pt
  },
every axis y label/.style={
below,
at={(0cm,1.9cm)},
  xshift=-3pt},
  y label style={rotate=90,anchor=south},
  width=4.6cm,
  height=3.8cm,  
  xmin=0,
  xmax=3,
  ymin=0,
  ymax=1]
\addplot graphics[xmin=0,xmax=3,ymin=0,ymax=1] {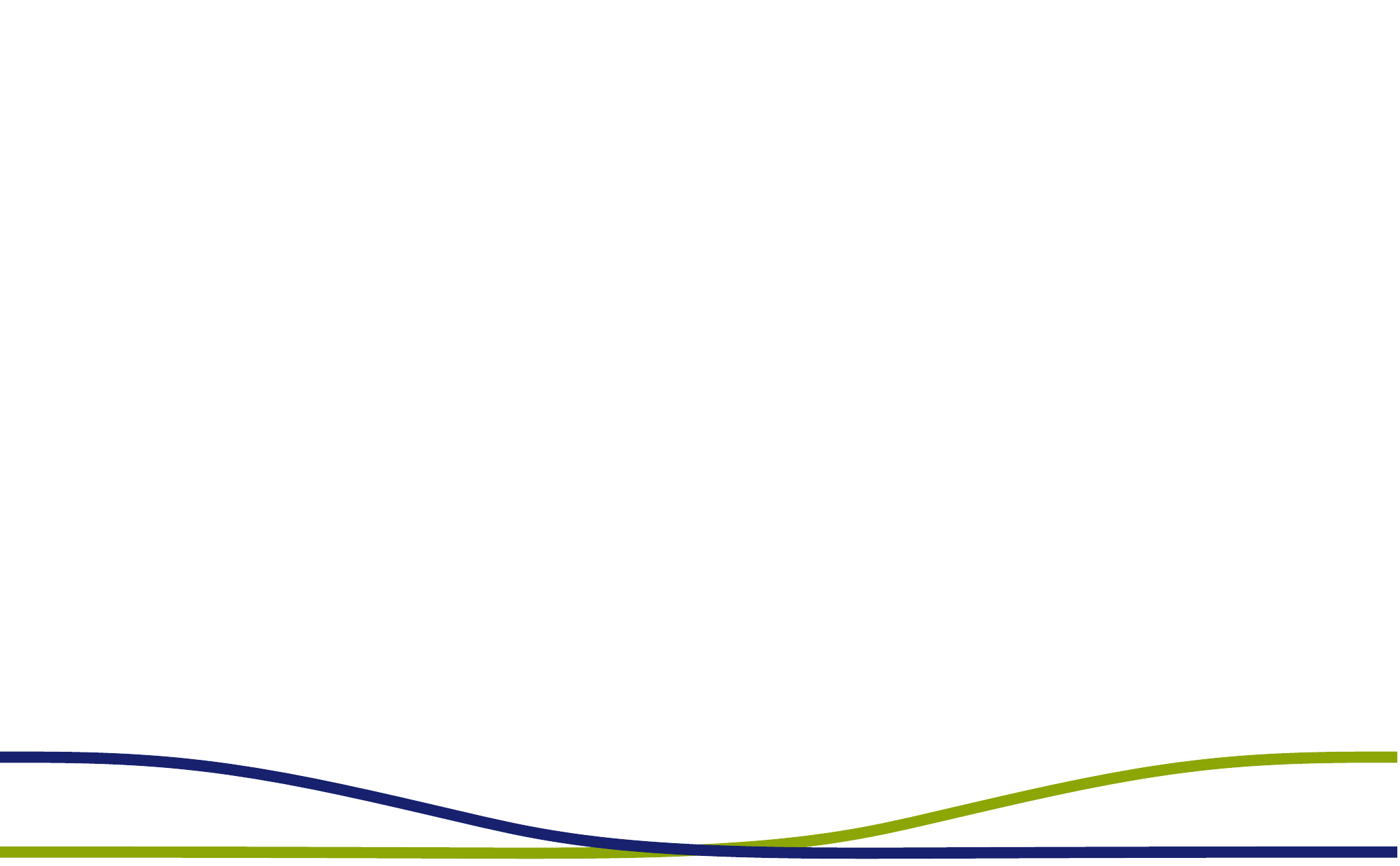};
\end{axis}
\end{tikzpicture} 
\quad\quad
\begin{tikzpicture}[scale=1]
\begin{axis}[
  tick pos=left,
  tick label style={font=\scriptsize},
          scale only axis,
  enlargelimits=false,
  xtick={0,3},
  ytick={0,1},
  xlabel={\small $t$},
  ylabel={\small $u(t)$},
  max space between ticks=50,
                minor x tick num=2,
                minor y tick num=5,  
every axis x label/.style={
below,
at={(2.3cm,0cm)},
  yshift=-3pt
  },
every axis y label/.style={
below,
at={(0cm,1.9cm)},
  xshift=-3pt},
  y label style={rotate=90,anchor=south},
  width=4.6cm,
  height=3.8cm,  
  xmin=0,
  xmax=3,
  ymin=0,
  ymax=1]
\addplot graphics[xmin=0,xmax=3,ymin=0,ymax=1] {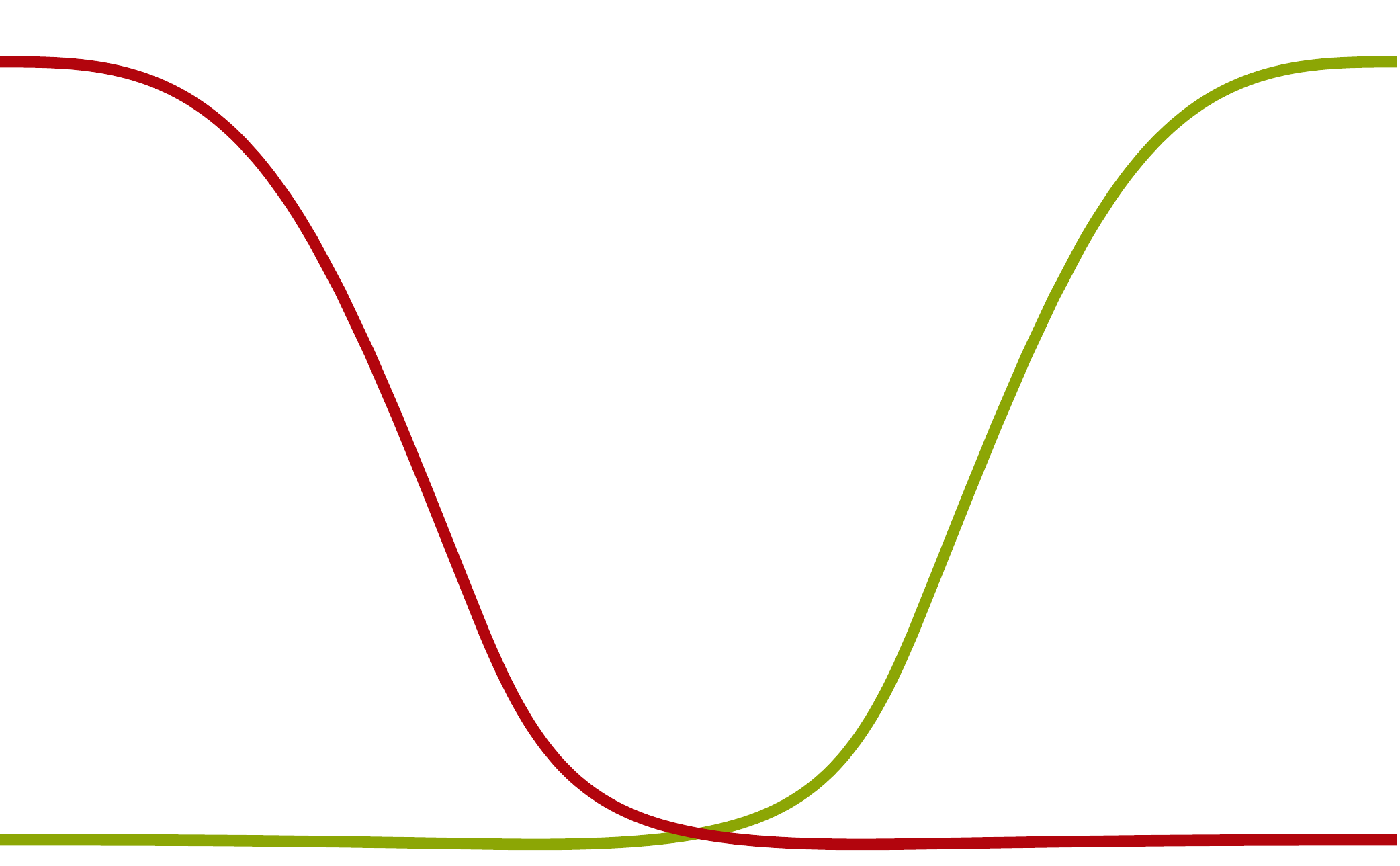};
\end{axis}
\end{tikzpicture}
\vspace{10pt}
\begin{tikzpicture}[scale=1]
\begin{axis}[
  tick pos=left,
  tick label style={font=\scriptsize},
          scale only axis,
  enlargelimits=false,
  xtick={0,3},
  ytick={0,1},
  xlabel={\small $t$},
  ylabel={\small $u(t)$},
  max space between ticks=50,
                minor x tick num=2,
                minor y tick num=5,  
every axis x label/.style={
below,
at={(2.3cm,0cm)},
  yshift=-3pt
  },
every axis y label/.style={
below,
at={(0cm,1.9cm)},
  xshift=-3pt},
  y label style={rotate=90,anchor=south},
  width=4.6cm,
  height=3.8cm,  
  xmin=0,
  xmax=3,
  ymin=0,
  ymax=1]
\addplot graphics[xmin=0,xmax=3,ymin=0,ymax=1] {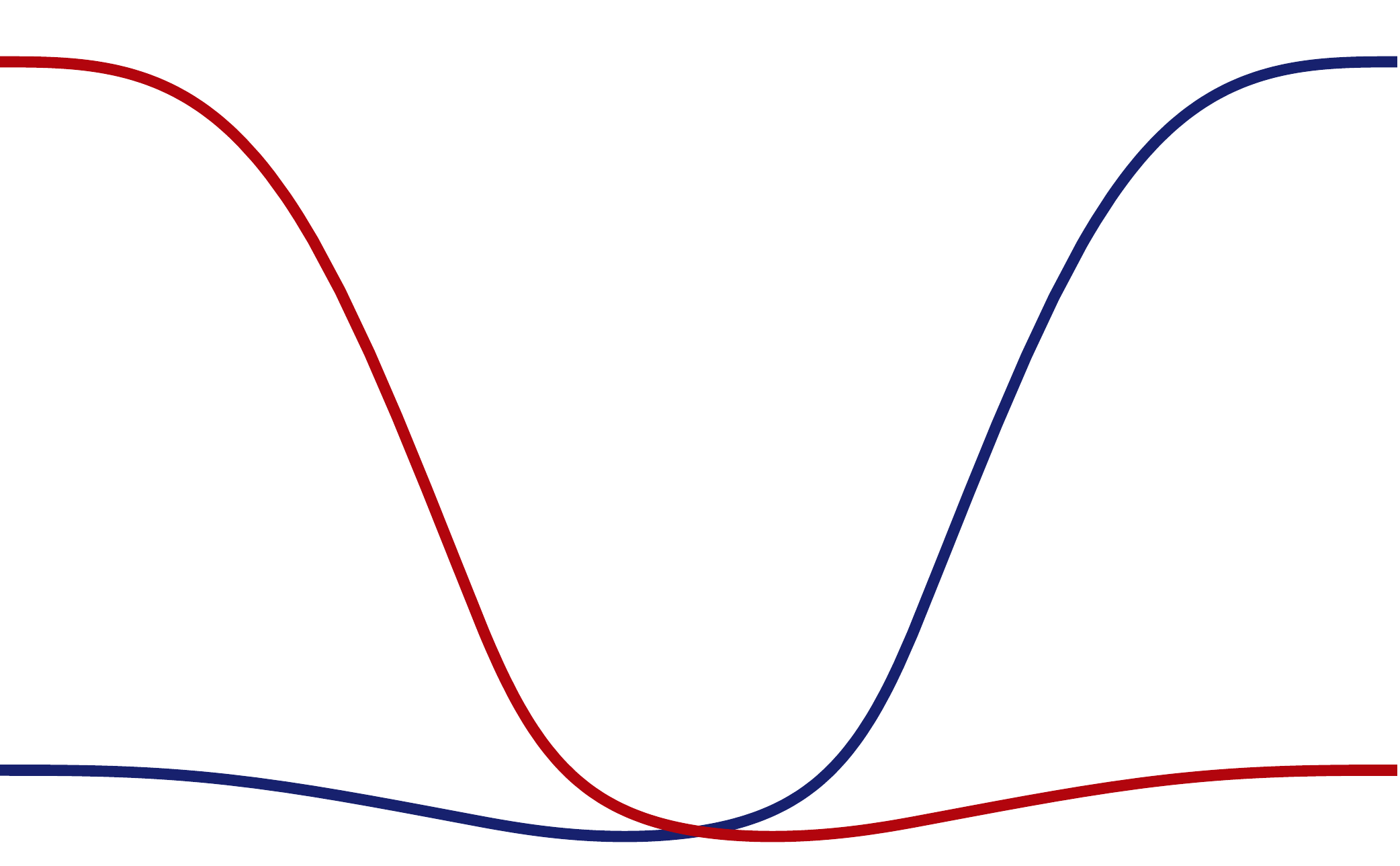};
\end{axis}
\end{tikzpicture} 
\quad\quad
\begin{tikzpicture}[scale=1]
\begin{axis}[
  tick pos=left,
  tick label style={font=\scriptsize},
          scale only axis,
  enlargelimits=false,
  xtick={0,3},
  ytick={0,1},
  xlabel={\small $t$},
  ylabel={\small $u(t)$},
  max space between ticks=50,
                minor x tick num=2,
                minor y tick num=5,  
every axis x label/.style={
below,
at={(2.3cm,0cm)},
  yshift=-3pt
  },
every axis y label/.style={
below,
at={(0cm,1.9cm)},
  xshift=-3pt},
  y label style={rotate=90,anchor=south},
  width=4.6cm,
  height=3.8cm,  
  xmin=0,
  xmax=3,
  ymin=0,
  ymax=1]
\addplot graphics[xmin=0,xmax=3,ymin=0,ymax=1] {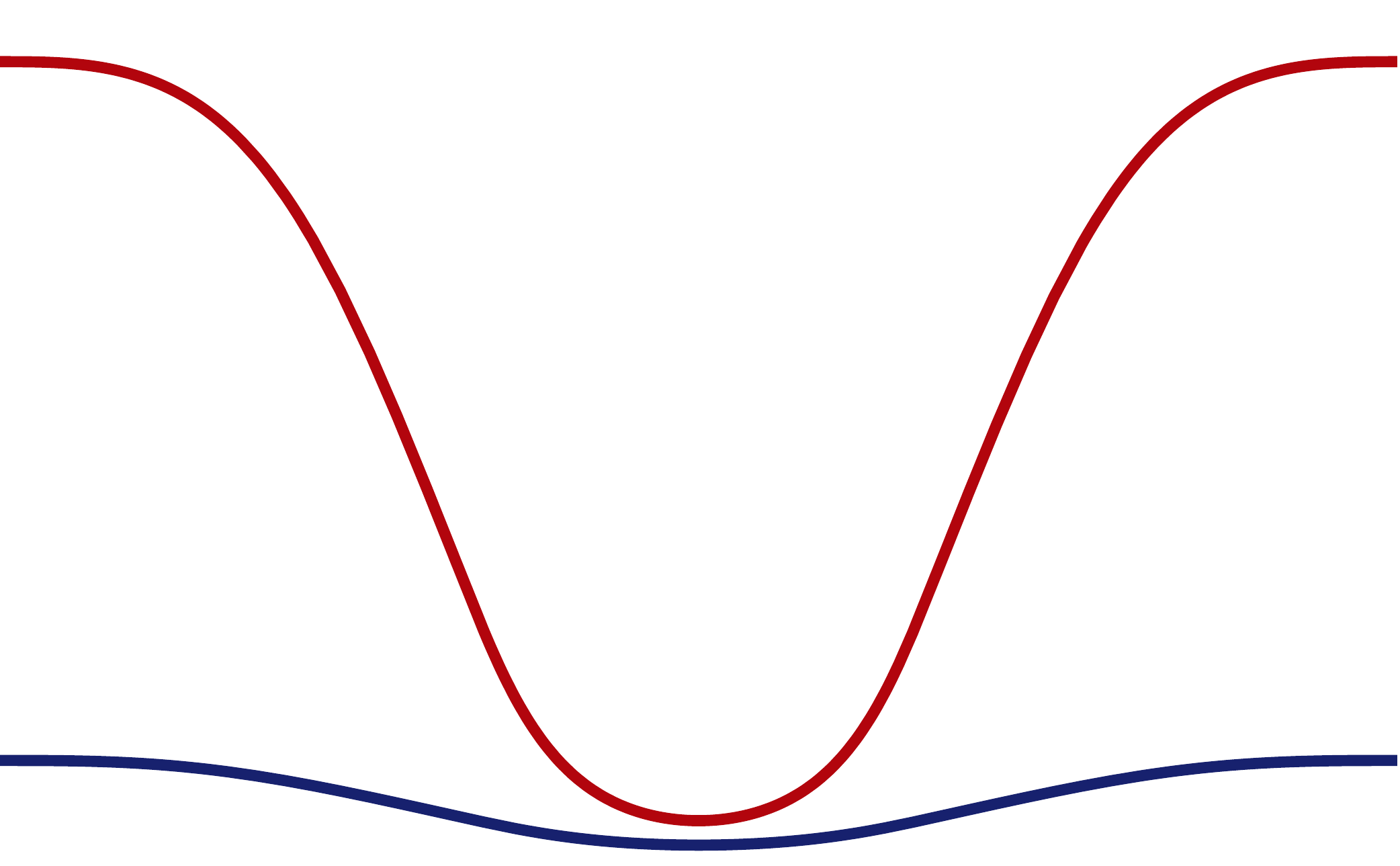};
\end{axis}
\end{tikzpicture} 
\caption{Graphs of eight positive solutions to Neumann problem \eqref{eq-ex} for $\lambda=20$ and $\mu=500$.}        
\label{fig-1}
\end{figure}

\medskip

Another point of view in the study on the number of positive solutions to the Neumann problem \eqref{eq-ex} consists in a qualitative bifurcation analysis. In particular, in our context we analyze bifurcation diagrams by fixing $\lambda$ and taking $\mu$ as the bifurcation parameter. We are going to discuss two different situations characterized by the choice of $\lambda$: on the one hand we illustrate the result of Theorem~\ref{th-main} for $\lambda$ sufficiently large, on the other hand we argue what happens in the other cases by means of numerical experiments illustrating some different and observable behaviors.

In this second part of the discussion we describe the bifurcation diagrams of some numerical experiments where we plot the values of $\mu$ in the horizontal axis versus the values of $u(0)$ in the vertical one.
In the bifurcation diagram of Figure~\ref{fig-2} we take the parameter $\lambda=20$ and consider the parameter $\mu$ ranging between $\mu=0$ and $\mu=600$. 
Actually, increasing $\mu$, we exhibit numerical evidence of the presence of eight positive (nontrivial) solutions to problem \eqref{eq-ex} in view of the existence of five bifurcation points: two transcritical, one pitchfork and two turning points. In particular, we observe that branches starting from the two turning points appear for $\mu>300$ (see also Figure~\ref{fig-2bis} for the magnification of the branches with ``large'' values of $u(0)$). This case covers the example treated in Figure~\ref{fig-1} and clearly  presents graphically the result stated in Theorem~\ref{th-main}. Indeed, in our framework the choice of $\lambda=20$ guarantees the applicability of Proposition~\ref{prop-2.2} (and symmetrically of Proposition~\ref{prop-2.4}), as suggested by the deformation of $X_{\mathopen{[}0,1\mathclose{]}}$ through the Poincar\'e map $\Phi_{0}^{\sigma}$ shown in the phase-plane plot in Figure~\ref{fig-2}. Moreover, by taking $\mu>\mu^{*}(20)$, with $\mu^{*}(20)>400$, the bifurcation diagram depicts the existence of eight positive (nontrivial) solutions to problem \eqref{eq-ex}.

Now, we investigate on quantitative aspects by decreasing the parameter $\lambda$. More precisely, we deal with some choices of $\lambda$ for which Proposition~\ref{prop-2.2} and Proposition~\ref{prop-2.4} do not apply, since the continuum $\Phi^{\sigma}_{0}(X_{\mathopen{[}0,1\mathclose{]}})$ does not intersect $\{0\}\times\mathopen{]}0,-\infty\mathclose{[}$. For that reason, we discuss the following cases: $\lambda=4$, $\lambda=8$ and $\lambda=10$ (cf.~the phase-plane plots in Figure~\ref{fig-3}, Figure~\ref{fig-4}, Figure~\ref{fig-5}, respectively).

For $\lambda=4$ numerical simulations show the existence of a branch (with a turning point) connecting two transcritical point. In addition to this configuration, for $\lambda=8$ and $\lambda=10$, a isola appears around the branch. In population dynamics it is not surprising to have a configuration of this type, as computed in \cite{LGMMTe-14,LGTe-14,LGTeZa-14,So-17}.
Moreover, we observe that, by increasing $\lambda$ in the range $\mathopen{[}8,10\mathclose{]}$, the isola is subjected to stretching and so, from the previous discussion for $\lambda=20$, these experiments suggest breaking points in the range $\mathopen{[}10,20\mathclose{]}.$
On account of the bifurcation diagrams, we remark that in each of the above cases there exist ranges $\mathopen{]}m_{0}(\lambda),m_{1}(\lambda)\mathclose{[}$ of the parameter $\mu$, with $0<m_{0}(\lambda)<m_{1}(\lambda)$, where problem \eqref{eq-ex} has at least two (nontrivial) solutions. 
The fact that $m_{0}(\lambda)>0$ is in accord with the condition for the existence of positive solutions to the Neumann problem associated with $(\mathscr{P}_{\lambda,\mu})$, namely 
$\mu > \lambda \int_{0}^{T} a^{+}(t)\,dt \, \big{/} \int_{0}^{T} a^{-}(t)\,dt$ (cf.~Theorem~\ref{th-LNS}).
Furthermore, in Figure~\ref{fig-3}, Figure~\ref{fig-4}, Figure~\ref{fig-5}, numerical simulations yield $m_{1}(\lambda)<+\infty$. At this point, in view of open problems presented in \cite{LoNaNi-13} and the achievement in \cite{BoGa-16dcds}, for problems with different kinds of diffusion rates and indefinite weight terms, the study of problem $(\mathscr{P}_{\lambda,\mu})$ for $\lambda$ small turns out to be a challenging issue. With this respect, the above discussion leads to the following conjecture.

\begin{figure}[!htb]
\centering
\begin{tikzpicture}[scale=1]
\begin{axis}[
  tick pos=left,
  tick label style={font=\scriptsize},
          scale only axis,
  enlargelimits=false,
  xtick={0,600},
  ytick={0,1},
  xlabel={\small $\mu$},
  ylabel={\small $u(0)$},
  max space between ticks=50,
                minor x tick num=5,
                minor y tick num=5,  
every axis x label/.style={
below,
at={(3.5cm,0cm)},
  yshift=-3pt
  },
every axis y label/.style={
below,
at={(0cm,2.5cm)},
  xshift=-3pt},
  y label style={rotate=90,anchor=south},
  width=7cm,
  height=5cm,  
  xmin=0,
  xmax=600,
  ymin=0,
  ymax=1]
\addplot graphics[xmin=0,xmax=600,ymin=-0.01,ymax=1.01] {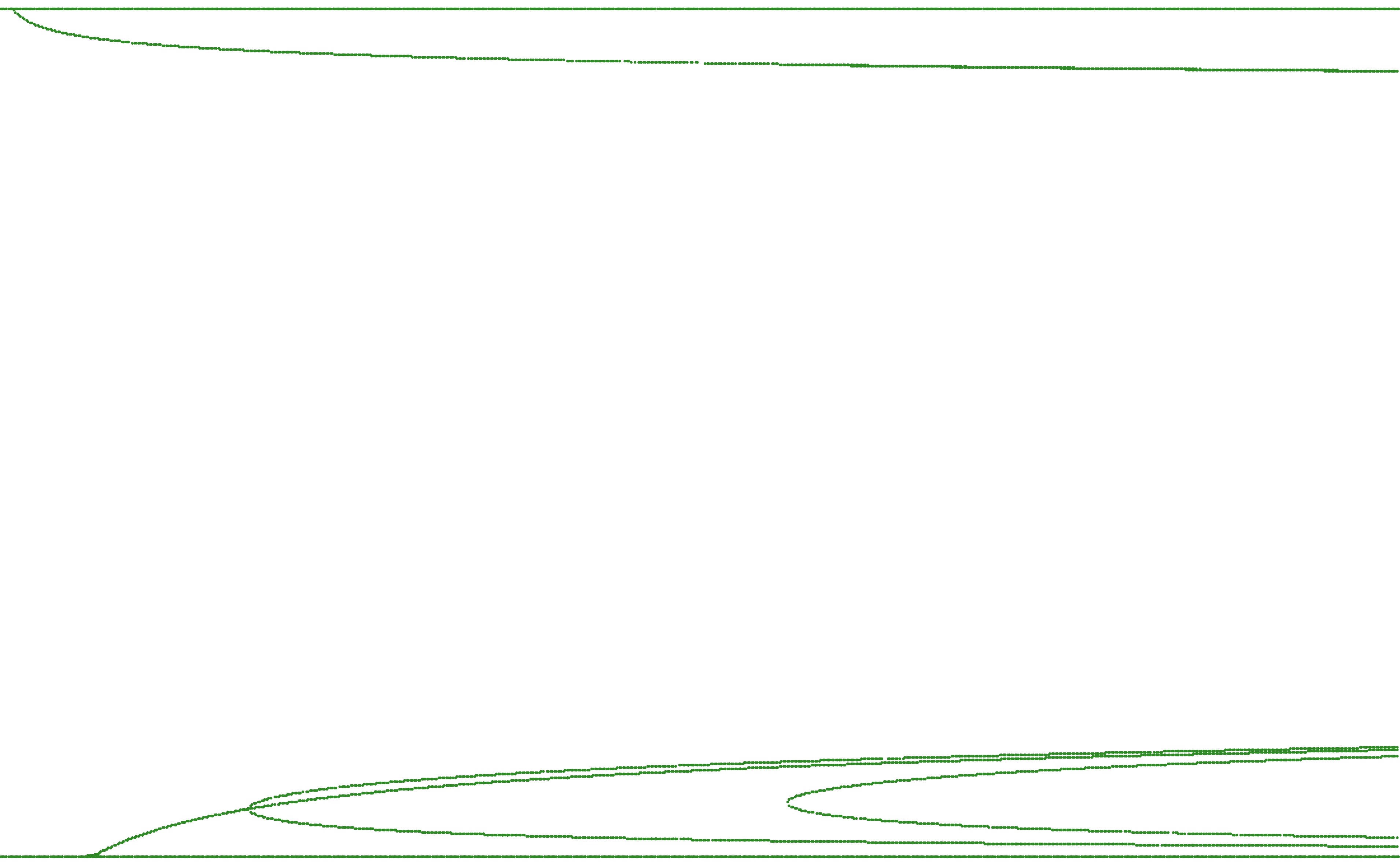};
\end{axis}
\end{tikzpicture}
\hspace{-0.2cm}
\begin{tikzpicture}[scale=1,baseline={(0,-1.35)}]
\begin{axis}[axis on top,
  tick pos=left,
  tick label style={font=\scriptsize},
  set layers,
    axis line style={on layer=axis foreground},% not working
    scale only axis,
    grid,
  enlargelimits=false,
  xtick={0,1},
  ytick={0},
  xlabel={\small $u$},
  ylabel={\small $u'$},
  every axis x label/.style={
below,
at={(1.5cm,0cm)},
  yshift=-3pt
  },
every axis y label/.style={
below,
at={(0cm,1.5cm)},
  xshift=-3pt},
  y label style={rotate=90,anchor=south},
  width=3cm, 
  height=3cm,
  xmin=-0.2,
  xmax=1.2,
  ymin=-1.6,
  ymax=0.2] 
\addplot graphics[xmin=-0.22,xmax=1.02,ymin=-1.6,ymax=0.21] {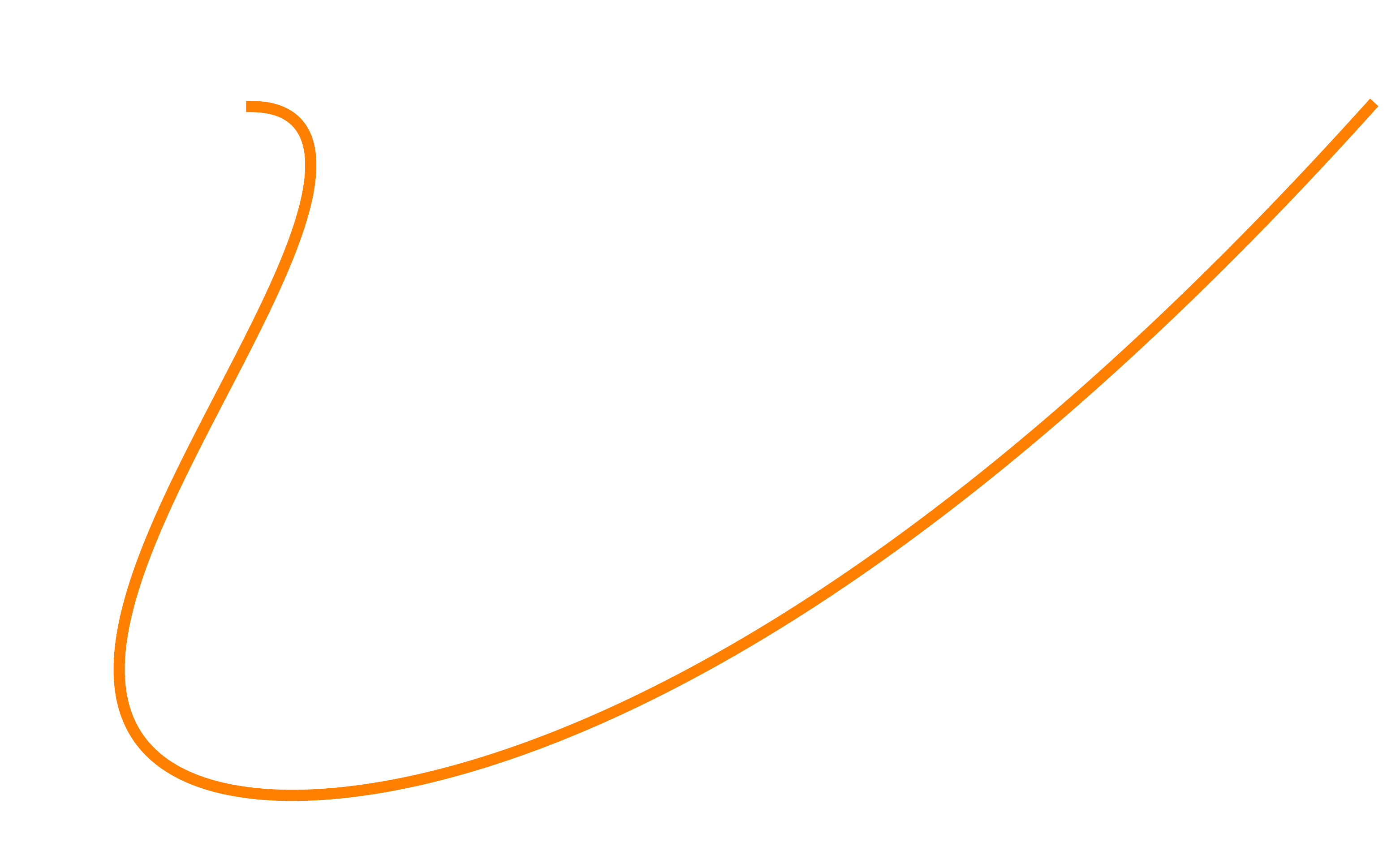};
\end{axis}
\end{tikzpicture} 
\caption{For the Neumann problem \eqref{eq-ex} with $\lambda=20$, bifurcation diagram with bifurcation parameter $\mu\in\mathopen{[}0,600\mathclose{]}$ (on the left) and deformation of $X_{\mathopen{[}0,1\mathclose{]}}$ through the Poincar\'e map $\Phi_{0}^{\sigma}$ (on the right).}        
\label{fig-2}
\end{figure}

\begin{figure}[!h]
\centering
\begin{tikzpicture}[scale=1]
\begin{axis}[
  tick pos=left,
  tick label style={font=\scriptsize},
          scale only axis,
  enlargelimits=false,
  xtick={300,600},
  ytick={0.925,0.935},
   yticklabel style={/pgf/number format/fixed,
                     /pgf/number format/precision=3},
  xlabel={\small $\mu$},
  ylabel={\small $u(0)$},
  max space between ticks=50,
                minor x tick num=2,
                minor y tick num=0,  
every axis x label/.style={
below,
at={(3.5cm,0cm)},
  yshift=-3pt
  },
every axis y label/.style={
below,
at={(0cm,2.5cm)},
  xshift=-3pt},
  y label style={rotate=90,anchor=south},
  width=7cm,
  height=5cm,
  xmin=300,
  xmax=600,
  ymin=0.925,
  ymax=0.935]
\addplot graphics[xmin=300,xmax=600,ymin=0.925,ymax=0.935] {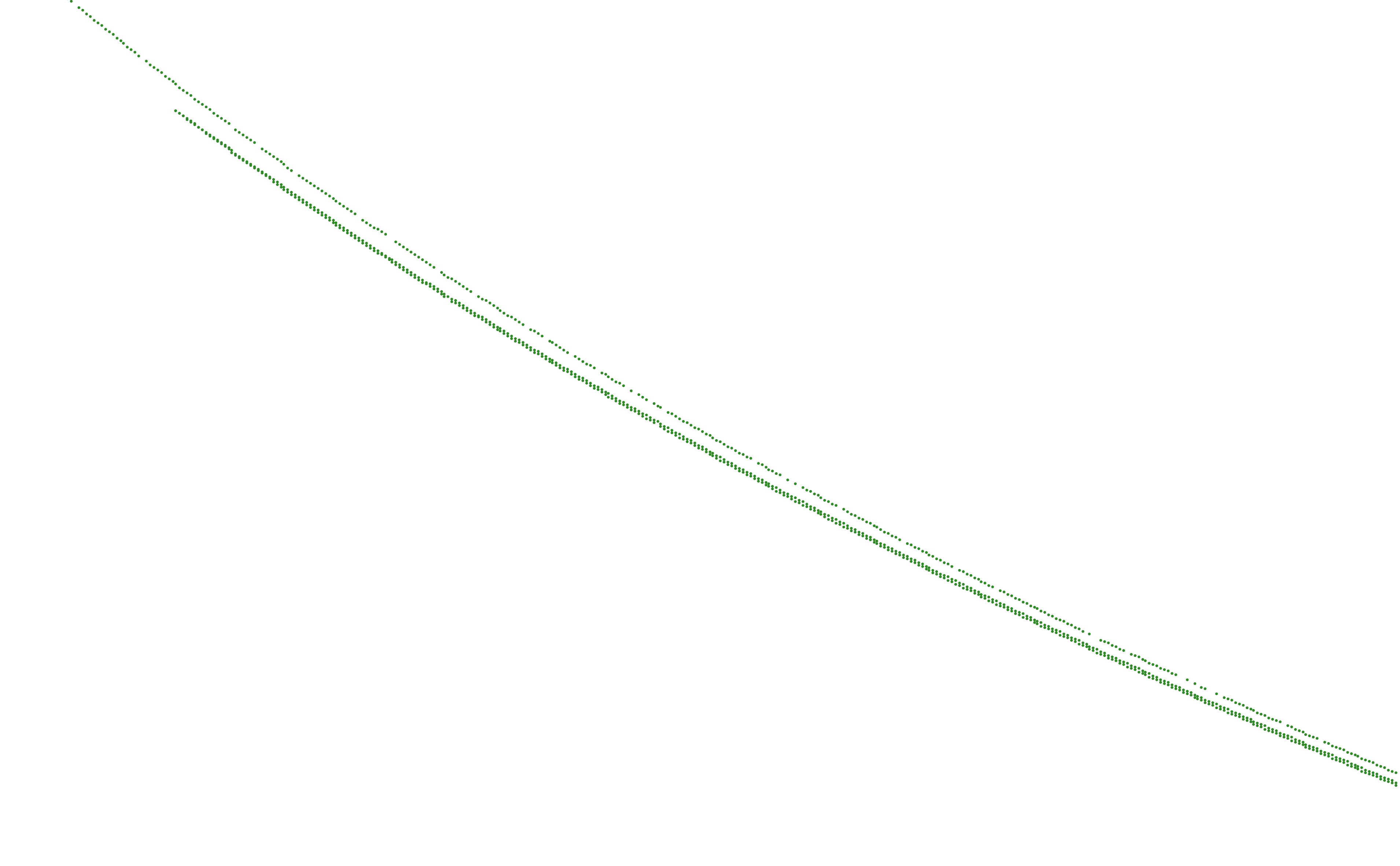};
\end{axis}
\end{tikzpicture} 
\caption{Detail of the bifurcation diagram for the Neumann problem \eqref{eq-ex} with $\lambda=20$ and bifurcation parameter $\mu\in\mathopen{[}300,600\mathclose{]}$ for $u(0)\in\mathopen{[}0.925,0.935\mathclose{]}$.}        
\label{fig-2bis}
\end{figure}

\begin{figure}[h]
\centering
\begin{tikzpicture}[scale=1]
\begin{axis}[
  tick pos=left,
  tick label style={font=\scriptsize},
          scale only axis,
  enlargelimits=false,
  xtick={0,20},
  ytick={0,1},
  xlabel={\small $\mu$},
  ylabel={\small $u(0)$},
  max space between ticks=50,
                minor x tick num=3,
                minor y tick num=5,  
every axis x label/.style={
below,
at={(3.5cm,0cm)},
  yshift=-3pt
  },
every axis y label/.style={
below,
at={(0cm,2.5cm)},
  xshift=-3pt},
  y label style={rotate=90,anchor=south},
  width=7cm,
  height=5cm,
  xmin=0,
  xmax=20,
  ymin=0,
  ymax=1]
\addplot graphics[xmin=0,xmax=20,ymin=-0.01,ymax=1.01] {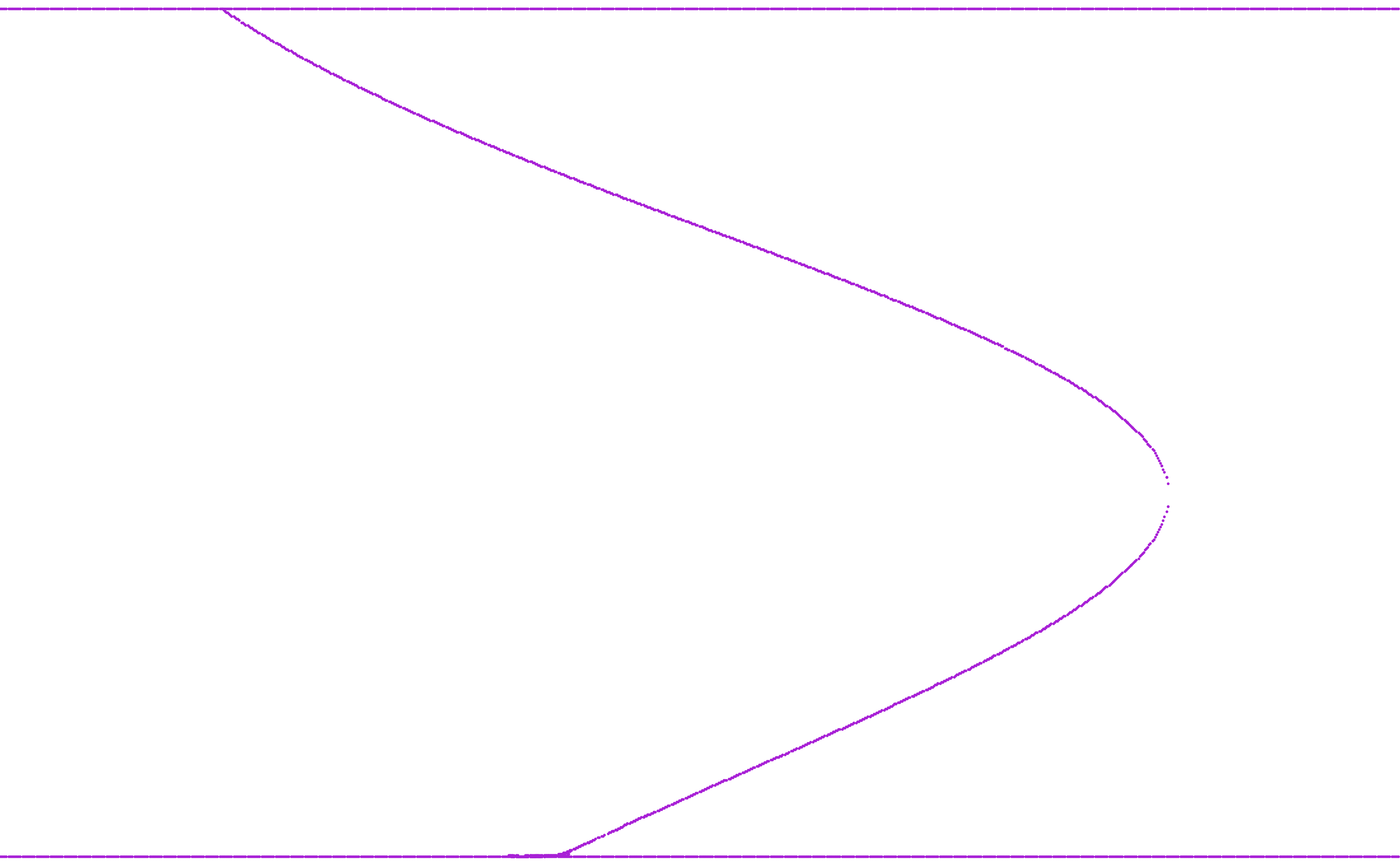};
\end{axis}
\end{tikzpicture}
\hspace{-0.2cm}
\begin{tikzpicture}[scale=1,baseline={(0,-1.35)}]
\begin{axis}[axis on top,
  tick pos=left,
  tick label style={font=\scriptsize},
  set layers,
    axis line style={on layer=axis foreground},% not working
    scale only axis,
    grid,
  enlargelimits=false,
  xtick={0,1},
  ytick={0},
  xlabel={\small $u$},
  ylabel={\small $u'$},
  every axis x label/.style={
below,
at={(1.5cm,0cm)},
  yshift=-3pt
  },
every axis y label/.style={
below,
at={(0cm,1.5cm)},
  xshift=-3pt},
  y label style={rotate=90,anchor=south},
  width=3cm, 
  height=3cm,
  xmin=-0.2,
  xmax=1.2,
  ymin=-1.6,
  ymax=0.2] 
\addplot graphics[xmin=-0.22,xmax=1.02,ymin=-1.6,ymax=0.21] {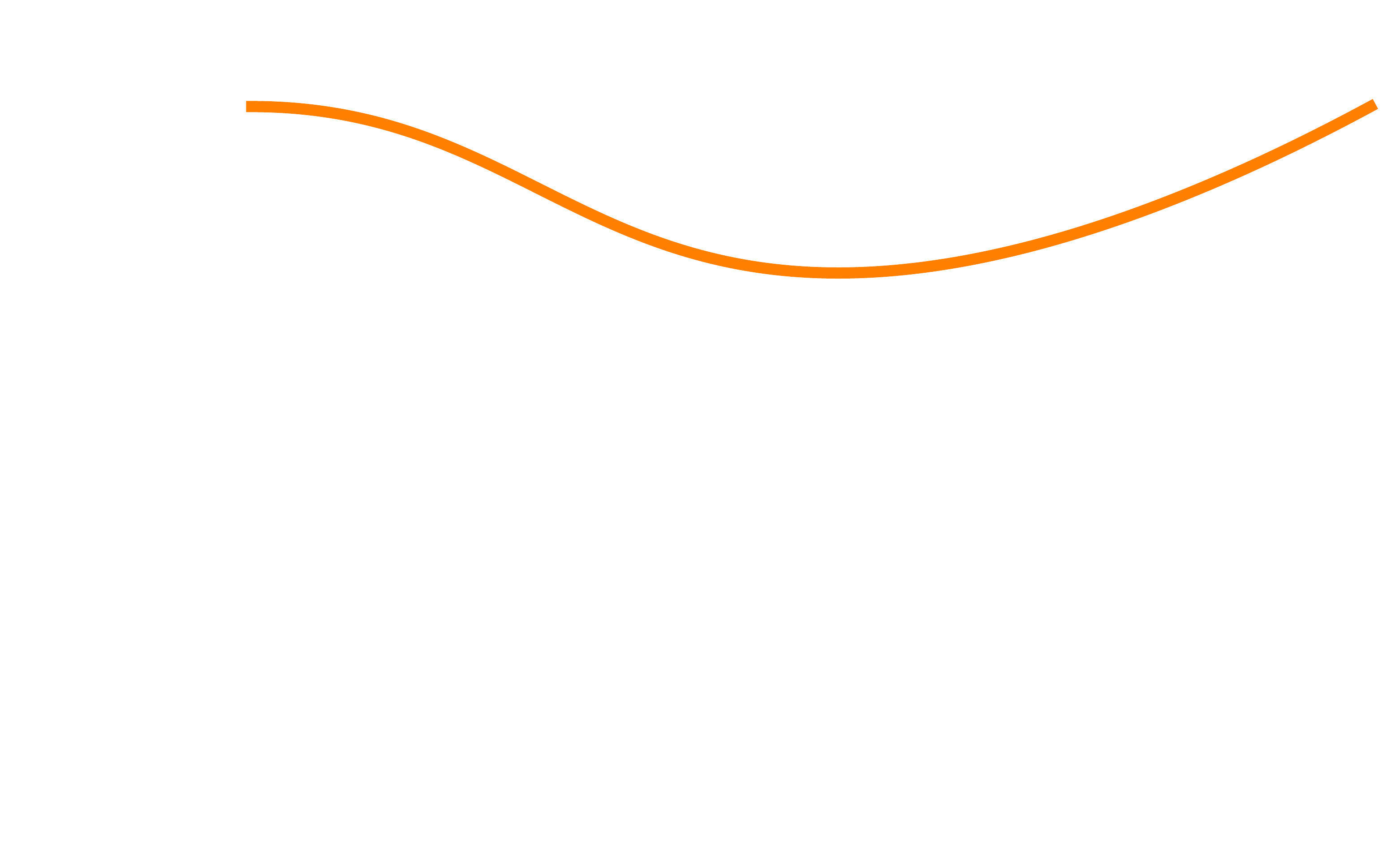};
\end{axis}
\end{tikzpicture} 
\caption{For the Neumann problem \eqref{eq-ex} with $\lambda=4$, bifurcation diagram with bifurcation parameter $\mu\in\mathopen{[}0,20\mathclose{]}$ (on the left) and deformation of $X_{\mathopen{[}0,1\mathclose{]}}$ through the Poincar\'e map $\Phi_{0}^{\sigma}$ (on the right).}        
\label{fig-3}
\end{figure}

\begin{figure}[!h]
\centering
\begin{tikzpicture}[scale=1]
\begin{axis}[
  tick pos=left,
  tick label style={font=\scriptsize},
          scale only axis,
  enlargelimits=false,
  xtick={0,200},
  ytick={0,1},
  xlabel={\small $\mu$},
  ylabel={\small $u(0)$},
  max space between ticks=50,
                minor x tick num=3,
                minor y tick num=5,  
every axis x label/.style={
below,
at={(3.5cm,0cm)},
  yshift=-3pt
  },
every axis y label/.style={
below,
at={(0cm,2.5cm)},
  xshift=-3pt},
  y label style={rotate=90,anchor=south},
  width=7cm,
  height=5cm,
  xmin=0,
  xmax=200,
  ymin=0,
  ymax=1]
\addplot graphics[xmin=0,xmax=200,ymin=-0.01,ymax=1.01] {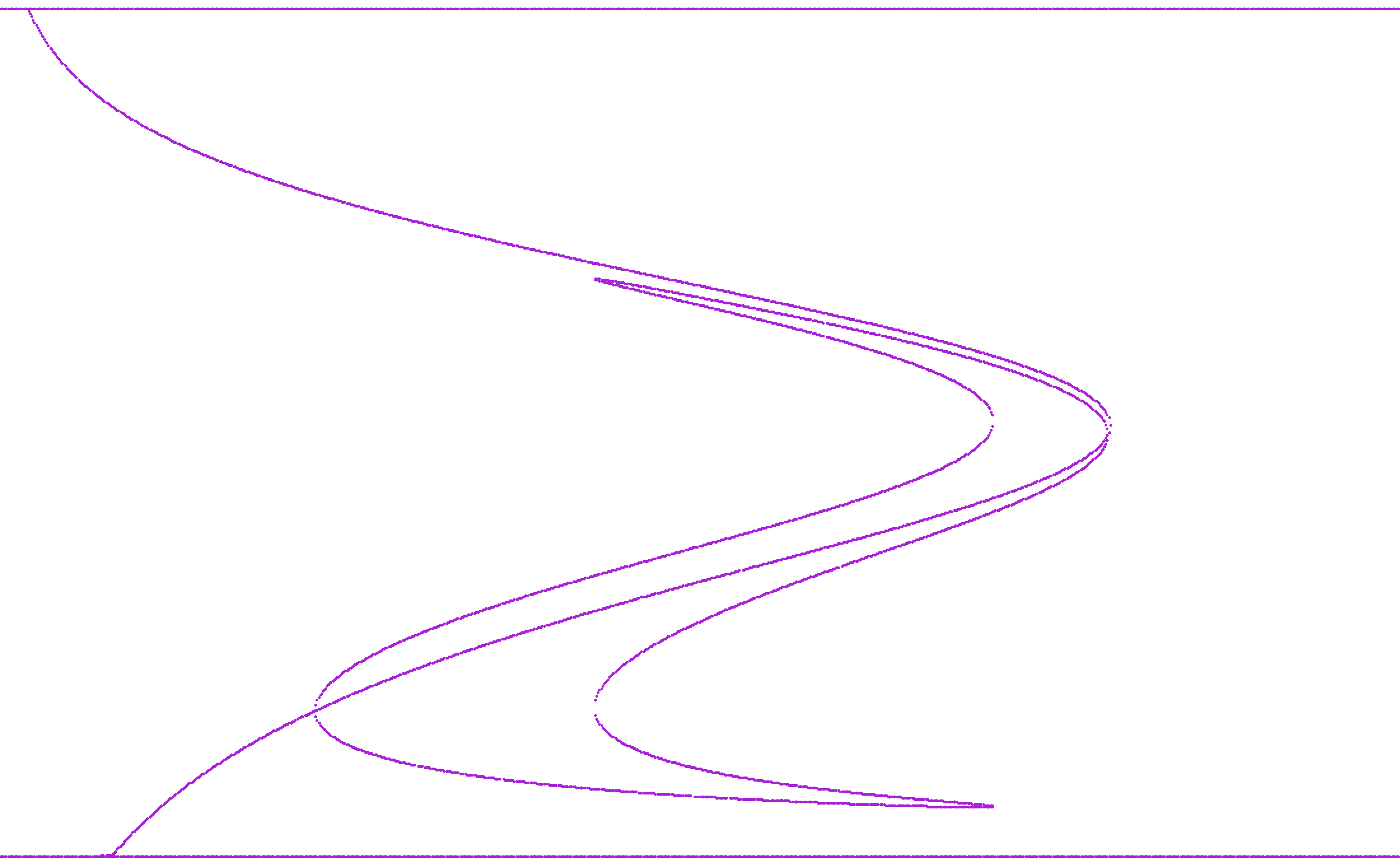};
\end{axis}
\end{tikzpicture}
\hspace{-0.2cm}
\begin{tikzpicture}[scale=1,baseline={(0,-1.35)}]
\begin{axis}[axis on top,
  tick pos=left,
  tick label style={font=\scriptsize},
  set layers,
    axis line style={on layer=axis foreground},% not working
    scale only axis,
    grid,
  enlargelimits=false,
  xtick={0,1},
  ytick={0},
  xlabel={\small $u$},
  ylabel={\small $u'$},
  every axis x label/.style={
below,
at={(1.5cm,0cm)},
  yshift=-3pt
  },
every axis y label/.style={
below,
at={(0cm,1.5cm)},
  xshift=-3pt},
  y label style={rotate=90,anchor=south},
  width=3cm, 
  height=3cm,
  xmin=-0.2,
  xmax=1.2,
  ymin=-1.6,
  ymax=0.2] 
\addplot graphics[xmin=-0.22,xmax=1.02,ymin=-1.6,ymax=0.21] {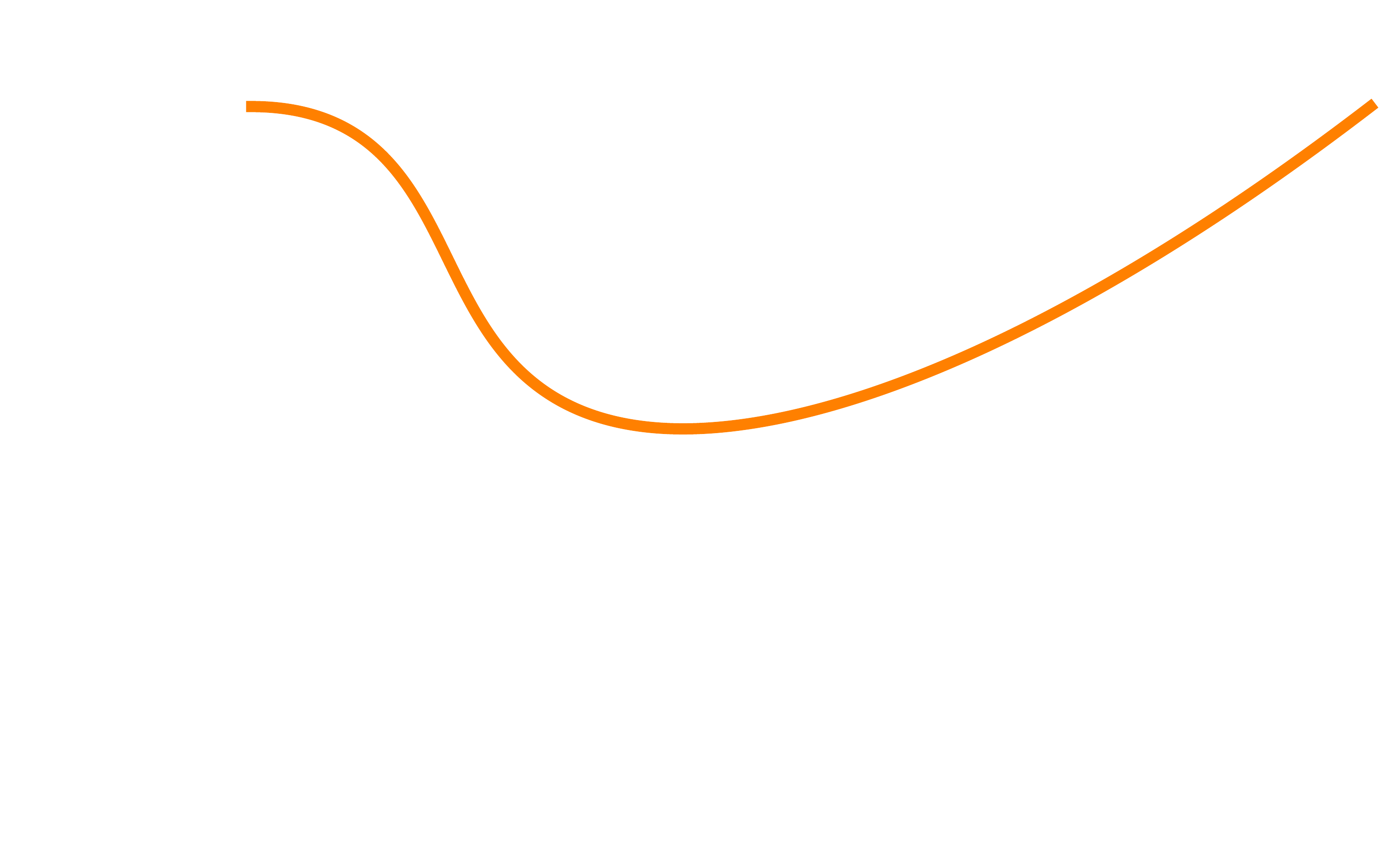};
\end{axis}
\end{tikzpicture} 
\caption{For the Neumann problem \eqref{eq-ex} with $\lambda=8$, bifurcation diagram with bifurcation parameter $\mu\in\mathopen{[}0,200\mathclose{]}$ (on the left) and deformation of $X_{\mathopen{[}0,1\mathclose{]}}$ through the Poincar\'e map $\Phi_{0}^{\sigma}$ (on the right).}        
\label{fig-4}
\end{figure}

\begin{figure}[h]
\centering
\begin{tikzpicture}[scale=1]
\begin{axis}[
  tick pos=left,
  tick label style={font=\scriptsize},
          scale only axis,
  enlargelimits=false,
  xtick={0,800},
  ytick={0,1},
  xlabel={\small $\mu$},
  ylabel={\small $u(0)$},
  max space between ticks=50,
                minor x tick num=7,
                minor y tick num=5,  
every axis x label/.style={
below,
at={(3.5cm,0cm)},
  yshift=-3pt
  },
every axis y label/.style={
below,
at={(0cm,2.5cm)},
  xshift=-3pt},
  y label style={rotate=90,anchor=south},
  width=7cm,
  height=5cm,
  xmin=0,
  xmax=800,
  ymin=0,
  ymax=1]
\addplot graphics[xmin=0,xmax=800,ymin=-0.01,ymax=1.01] {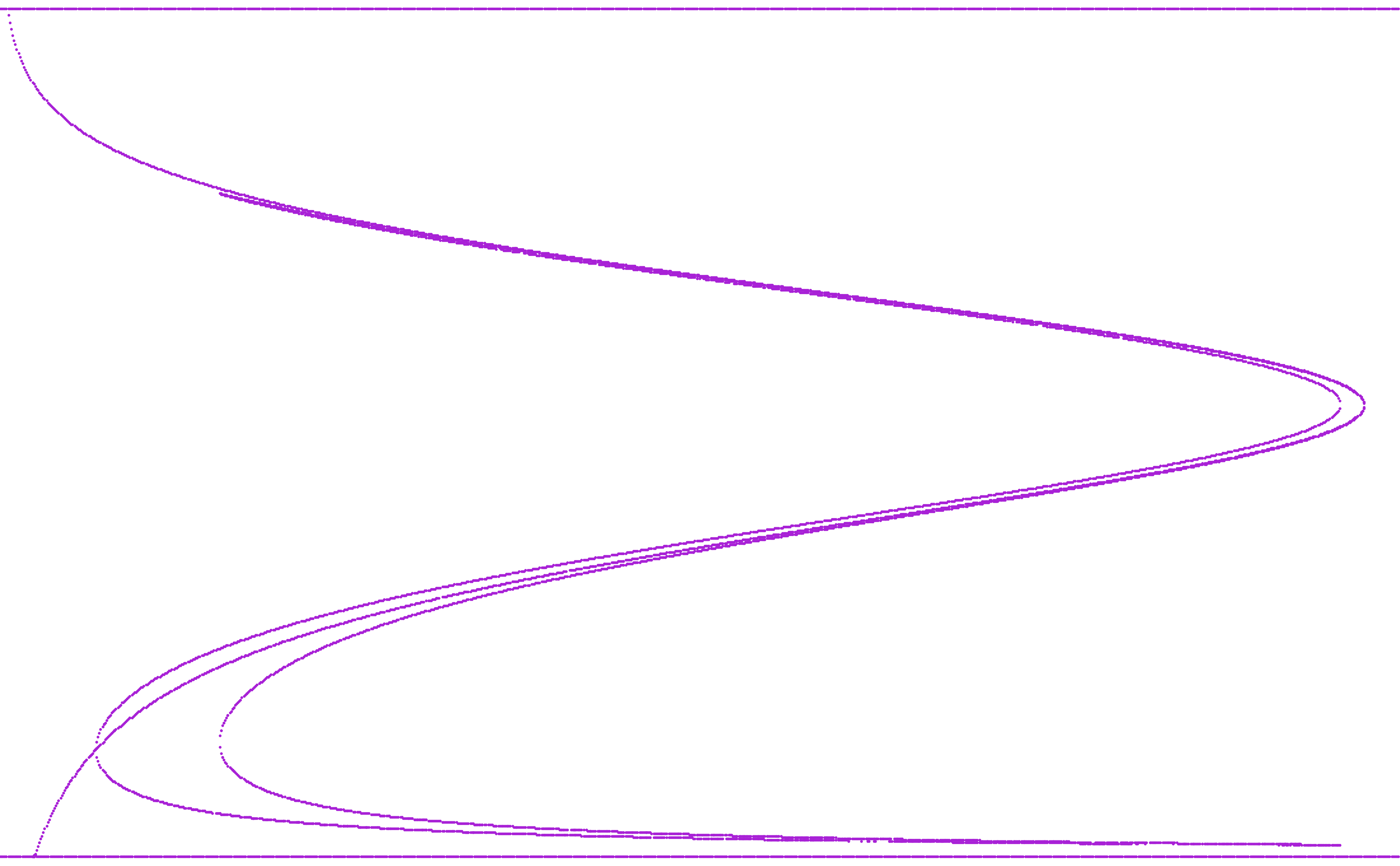};
\end{axis}
\end{tikzpicture}
\hspace{-0.2cm}
\begin{tikzpicture}[scale=1,baseline={(0,-1.35)}]
\begin{axis}[axis on top,
  tick pos=left,
  tick label style={font=\scriptsize},
  set layers,
    axis line style={on layer=axis foreground},% not working
    scale only axis,
    grid,
  enlargelimits=false,
  xtick={0,1},
  ytick={0},
  xlabel={\small $u$},
  ylabel={\small $u'$},
  every axis x label/.style={
below,
at={(1.5cm,0cm)},
  yshift=-3pt
  },
every axis y label/.style={
below,
at={(0cm,1.5cm)},
  xshift=-3pt},
  y label style={rotate=90,anchor=south},
  width=3cm, 
  height=3cm,
  xmin=-0.2,
  xmax=1.2,
  ymin=-1.6,
  ymax=0.2] 
\addplot graphics[xmin=-0.22,xmax=1.02,ymin=-1.6,ymax=0.21] {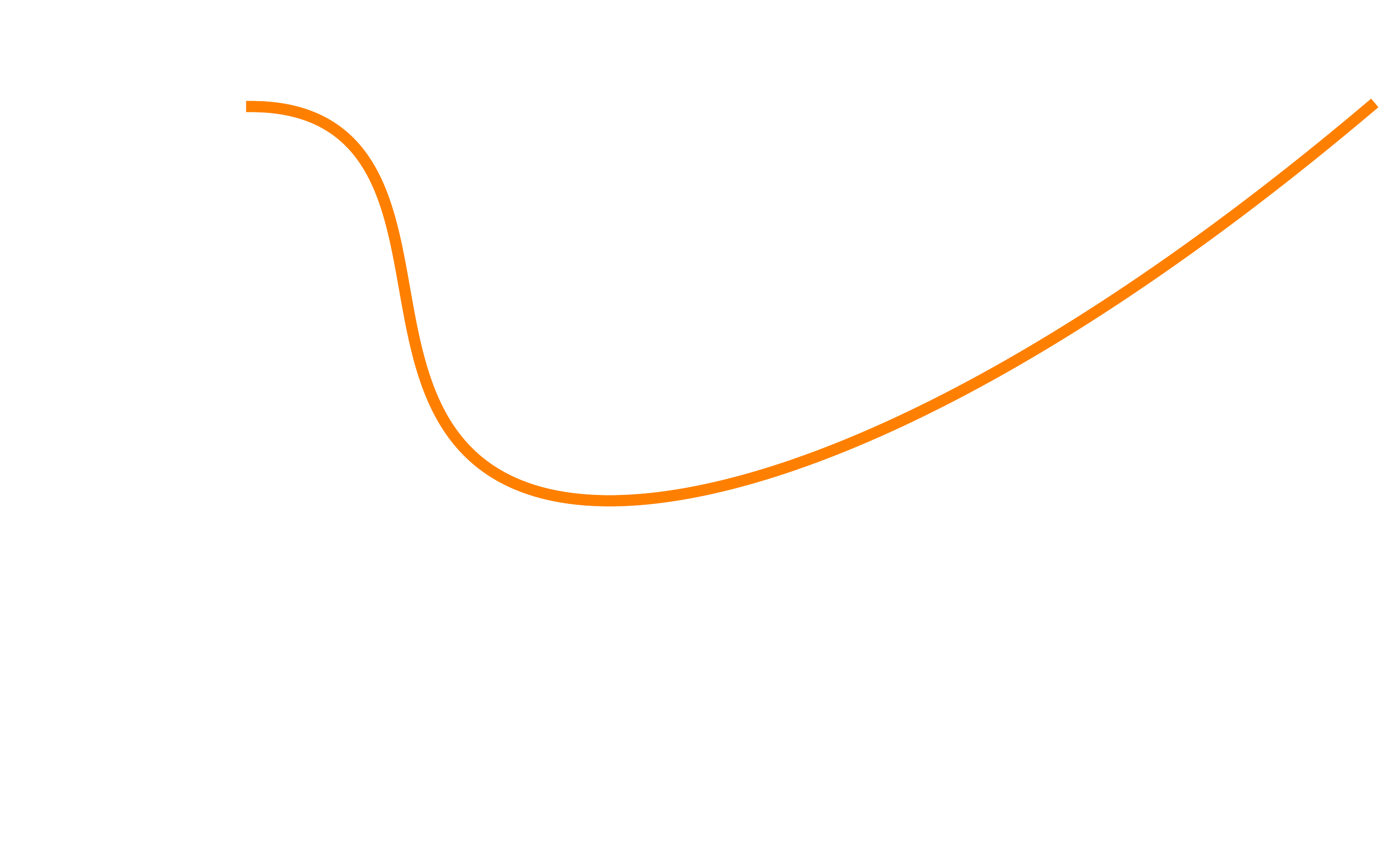};
\end{axis}
\end{tikzpicture} 
\caption{For the Neumann problem \eqref{eq-ex} with $\lambda=10$, bifurcation diagram with bifurcation parameter $\mu\in\mathopen{[}0,800\mathclose{]}$ (on the left) and deformation of $X_{\mathopen{[}0,1\mathclose{]}}$ through the Poincar\'e map $\Phi_{0}^{\sigma}$ (on the right).}        
\label{fig-5}
\end{figure}

\begin{conjecture}
Let $g \colon \mathopen{[}0,1\mathclose{]} \to \mathbb{R}^{+}$ be a locally Lipschitz continuous function satisfying $(g_{*})$ and $(g_{0})$.
Let $a \colon \mathopen{[}0,T\mathclose{]} \to \mathbb{R}$ be an $L^{1}$-function satisfying $(a_{*})$. Then, there exists $\lambda_{*}>0$ such that for each $\lambda\in\mathopen{]}0,\lambda_{*}\mathclose{]}$ there exists $\mu_{*}(\lambda)>0$ such that for every $\mu>\mu_{*}(\lambda)$ problem $(\mathscr{P}_{\lambda,\mu})$ has no positive solutions.
\end{conjecture}

\medskip

We conclude this section with a further open problem motivated by the papers \cite{BoGoHa-05,FeZa-15jde,FeZa-17}, where $g(u)\sim u^{p}$, $p>1$, or the paper \cite{BoFeZa-17tams}, where $g(u)\sim u^{2}/(1+u^{2})$. In these works, the weight term has $m$ intervals of positivity separated by intervals of negativity which characterize the number of positive solutions. As previously observed with respect to problem \eqref{eq-ex} (with $m=2$), each solution exhibits the following three behaviors in each interval of positivity: being ``very small'', ``small'', ``large''. This ``combinatory'' produces the $8=3^{2}-1$ positive (nontrivial) solutions. Thus, in the same spirit of \cite{BoFeZa-17tams}, we introduce the following problem.

\begin{conjecture}\label{con-2}
Let $g \colon \mathopen{[}0,1\mathclose{]} \to \mathbb{R}^{+}$ be a locally Lipschitz continuous function satisfying $(g_{*})$ and $(g_{0})$.
Let $a \colon \mathopen{[}0,T\mathclose{]} \to \mathbb{R}$ be an $L^{1}$-function. Suppose that there exist $2m$ points
\begin{equation*}
0 = \tau_{1} < \sigma_{1} < \ldots < \tau_{i} < \sigma_{i} < \ldots < \tau_{m} <
\sigma_{m}  = T,
\end{equation*}
such that $a(t)\succ0$ on $\mathopen{[}\tau_{i},\sigma_{i}\mathclose{]}$ and $a(t)\prec0$ on $\mathopen{[}\sigma_{i},\tau_{i+1}\mathclose{]}.$
Then, there exists $\lambda^{*}>0$ such that for each $\lambda>\lambda^{*}$ there exists $\mu^{*}(\lambda)>0$ such that for every $\mu>\mu^{*}(\lambda)$ problem $(\mathscr{P}_{\lambda,\mu})$ has at least $3^{m}-1$ positive solutions.
\end{conjecture}

Lastly, our aim is to intuitively support the above statement. Let us consider a weight term $a(t)$ as in Conjecture~\ref{con-2}. 
From Proposition~\ref{prop-2.1} and Proposition~\ref{prop-2.2}, depending on the boundary conditions assumed, there exist three sub-intervals of $Y_{\geq 0}$ and three sub-intervals of $X_{\mathopen{[}0,1\mathclose{]}}$, respectively, which are deformed through the Poincar\'{e} map $\Phi_{0}^{\kappa}$, with $\kappa=\tau_{2}$, in three continua $\Gamma_{j}$, with $j=1,2,3$, which connect $\{0\}\times\mathopen{]}-\infty,0\mathclose{]}$ with $\{1\}\times\mathopen{]}0,+\infty\mathclose{[}$ (we stress that the proof of these propositions is valid for each $\kappa=\mathopen{]}\sigma_{1},\tau_{2}\mathclose{]}$).

Looking at numerical experiments we can say something more. Firstly, these three continua $\Gamma_{j}$ connect $\{0\}\times\mathopen{]}-\varepsilon,0\mathclose{]}$ with $\{1\}\times\mathopen{]}0,+\infty\mathclose{[}$, where $\varepsilon\to0$ as $\mu\to+\infty$. Secondly, we observe that the deformation of each $\Gamma_{i}$ through the Poincar\'e map $\Phi_{\tau_{2}}^{\sigma_{2}}$ produces a similar configuration of $\Phi_{0}^{\sigma_{1}}(Y_{\geq 0})$ (independently of the boundary condition at $t=0$). 

These properties suggest that, for each $j\in\{1,2,3\}$, the deformation of $\Phi_{\tau_{2}}^{\sigma_{2}}(\Gamma_{j})$ through the Poincar\'{e} map $\Phi_{\sigma_{2}}^{\tau_{3}}$ produces other three continua which connect $\{0\}\times\mathopen{]}-\infty,0\mathclose{]}$ with $\{1\}\times\mathopen{]}0,+\infty\mathclose{[}$. These last configurations are similar to the one of $\Phi_{0}^{\tau_{2}}(Y_{\geq0})$, and so, at this point, one could enter in a inductive procedure scheme. Indeed, by fixing $\kappa$ arbitrarily in the last interval of negativity $\mathopen{[}\sigma_{m-1},\tau_{m}\mathclose{]}$, by induction one should obtain the existence of $3^{m-1}$ subintervals of $Y_{\geq 0}$ or $3^{m-1}$ subintervals of $X_{\mathopen{[}0,1\mathclose{]}}$, respectively, that are then deformed through the Poincar\'{e} map $\Phi_{0}^{\kappa}$ in $3^{m-1}$ continua which connect $\{0\}\times\mathopen{]}-\infty,0\mathclose{]}$ with $\{1\}\times\mathopen{]}0,+\infty\mathclose{[}$. On the other hand, from an application of Proposition~\ref{prop-2.3} and Proposition~\ref{prop-2.4} we obtain $3$ subintervals of $Y_{\leq 0}$ or $3$ subintervals of $X_{\mathopen{[}0,1\mathclose{]}}$, respectively, that are deformed through the Poincar\'{e} map $\Phi_{T}^{\kappa}$ in $3$ continua which connect $\{0\}\times\mathopen{[}0,+\infty\mathclose{[}$ with $\{1\}\times\mathopen{]}-\infty,0\mathclose{[}$. Finally, the intersection points of the continua found above give $3^{m}-1$ positive solutions (not counting the intersection corresponding to the zero solution), where $m$ is the number of positive humps of the weight $a(t)$.

\appendix
\section{Appendix: technical lemmas}\label{appendix-A}

In this appendix, we collect some standard results concerning initial value problems associated with
equation 
\begin{equation}\label{eq-A.1}
u'' + \bigl{(} \lambda a^{+}(t) - \mu a^{-}(t) \bigr{)} g(u) = 0.
\end{equation}
More precisely, we present useful estimates for the solutions to associated Cauchy problems, in analogy with the analysis performed in \cite{BoZa-13,FeSo-NA,GaHaZa-03,GaHaZa-04}.

In the light of the applications in Section~\ref{section-2}, the more convenient environment to present these results is the phase-plane $(x,y)=(u,u')$. Accordingly, we deal with the planar system
\begin{equation*}
\begin{cases}
\, x' = y, \\
\, y' = - \bigl{(} \lambda a^{+}(t) - \mu a^{-}(t) \bigr{)} g(x).
\end{cases}
\leqno{(\mathcal{S}_{\lambda,\mu})}
\end{equation*}

Throughout this section we suppose that $a \colon \mathopen{[}0,T\mathclose{]} \to \mathbb{R}$ is an $L^{1}$-function satisfying $(a_{*})$, and $g \colon \mathopen{[}0,1\mathclose{]} \to \mathbb{R}^{+}$ is a locally Lipschitz continuous function satisfying $(g_{*})$ and $(g_{0})$.
As in Section~\ref{section-2}, we extend the function $g(s)$ continuously to the whole real line, by setting
\begin{equation*}
g(s)= 0, \quad \text{for } \, s\in\mathopen{]}-\infty,0\mathclose{[}\cup\mathopen{]}1,+\infty\mathclose{[}.
\end{equation*}
In this manner, any solution of a Cauchy problem associated with $(\mathcal{S}_{\lambda,\mu})$ is unique and globally defined on $\mathopen{[}0,T\mathclose{]}$.

Lastly, we introduce the following notation
\begin{equation*}
A^{\pm}(t', t''):=\int_{t'}^{t''}a^{\pm}(\xi) \,d\xi, \quad t', t''\in\mathopen{[}0,T\mathclose{]} \, \text{ with } t'\leq t'',
\end{equation*}
and
\begin{equation*}
g_{m}(\eta',\eta'') := \min_{s\in\mathopen{[}\eta',\eta''\mathclose{]}}g(s), \quad \eta',\eta''\in\mathopen{[}0,1\mathclose{]} \, \text{ with } \eta' < \eta''.
\end{equation*}

\subsection{Analysis in $\mathopen{[}0,\sigma\mathclose{]}$}\label{section-A.1}

We deal with system $(\mathcal{S}_{\lambda,\mu})$ in the interval $\mathopen{[}0,\sigma\mathclose{]}$, where it is written as
\begin{equation}\label{IVP-1}
\begin{cases}
\, x' = y, \\
\, y' = - \lambda a^{+}(t) g(x).
\end{cases}
\end{equation}

The following lemmas state that, for $\lambda>0$, the solution $(x(t),y(t))$ of \eqref{IVP-1} with initial conditions $(x(0),y(0))=(x_{0},y_{0})$ satisfies
\begin{itemize}
\item $(x(\sigma),y(\sigma))\in\mathopen{]}1,+\infty\mathclose{[}\times\mathopen{]}0,+\infty\mathclose{[}$, taking $(x_{0},y_{0})=(0,\delta)$ with $\delta>0$ sufficiently large (cf.~Lemma~\ref{lem-A1.2});

\item $(x(t),y(t))\in\mathopen{]}0,1\mathclose{[}\times\mathopen{]}0,+\infty\mathclose{[}$, for all $t\in\mathopen{]}0,\sigma{]}$, taking $(x_{0},y_{0})=(0,\delta)$ with $\delta>0$ sufficiently small (cf.~Lemma~\ref{lem-A1.1});

\item $(x(\sigma),y(\sigma))\in\mathopen{]}-\infty,0\mathclose{]}\times\mathopen{]}-\infty,0\mathclose{[}$, taking $(x_{0},y_{0})=(\nu,0)$ with $\nu\in\mathopen{]}0,1\mathclose{[}$ and $\lambda$ sufficiently large (cf.~Lemma~\ref{lem-A1.3} and Lemma~\ref{lem-A1.4});

\item $(x(t),y(t))$ belongs to a small angular region contained in $\mathopen{]}0,1\mathclose{[}\times\mathopen{]}-\infty,0\mathclose{[}$, for all $t\in\mathopen{]}0,\sigma{]}$, taking $(x_{0},y_{0})=(\nu,0)$ with $\nu>0$ sufficiently small (cf.~Lemma~\ref{lem-A1.5}).
\end{itemize}

\begin{lemma}\label{lem-A1.2}
Let $\lambda>0$. There exists $\tilde{\delta}>0$ such that for every $\delta \geq \tilde{\delta}$ the solution $(x(t),y(t))$ of \eqref{IVP-1} with $(x(0),y(0))=(0,\delta)$ satisfies $(x(\sigma),y(\sigma))\in\mathopen{]}1,+\infty\mathclose{[}\times\mathopen{]}0,+\infty\mathclose{[}$.
\end{lemma}

\begin{proof}
Let $\lambda>0$ and
\begin{equation}
\tilde{\delta} > \dfrac{1}{\sigma} + \lambda\|a^{+}\|_{L^{1}(0,\sigma)} \max_{s\in\mathopen{[}0,1\mathclose{]}} g(s).
\end{equation}
Let $\delta>\tilde{\delta}$ be fixed and $(x(t),y(t))$ be the solution of \eqref{IVP-1} with initial conditions $(x(0),y(0))=(0,\delta)$.
From
\begin{equation*}
y(t) = \delta - \lambda \int_{0}^{t} a^{+}(\xi)g(x(\xi)) \, d\xi, \quad \text{for all } t\in\mathopen{[}0,\sigma\mathclose{]},
\end{equation*}
we have
\begin{equation*}
x(\sigma) \geq \sigma \delta - \lambda \sigma \|a^{+}\|_{L^{1}(0,\sigma)} \max_{s\in\mathopen{[}0,1\mathclose{]}} g(s) > 1
\end{equation*}
and $y(\sigma)>0$. Then the lemma is proved.
\end{proof}

\begin{lemma}\label{lem-A1.1}
Let $\lambda>0$ and let $\delta_{0}\in\mathopen{]}0,1/\sigma\mathclose{[}$ be such that
\begin{equation*}
\biggl{(}\dfrac{\pi}{2\sigma}\biggr{)}^{2} - \lambda a^{+}(t) \sup_{s\in\mathopen{]}0,\delta_{0}\sigma\mathclose{]}}\dfrac{g(s)}{s} > 0, \quad \text{for a.e.~$t\in\mathopen{[}0,\sigma\mathclose{]}$}.
\end{equation*}
If $(x(t),y(t))$ is the solution of \eqref{IVP-1} with initial conditions $(x(0),y(0))=(0,\delta_{0})$, then $(x(t),y(t))\in\mathopen{]}0,1\mathclose{[}\times\mathopen{]}0,+\infty\mathclose{[}$, for all $t\in\mathopen{]}0,\sigma{]}$.
\end{lemma}

\begin{proof}
Let $\lambda,\delta_{0}$ be fixed as in the statement. Let $(x(t),y(t))$ be the solution of \eqref{IVP-1} with initial conditions $(x(0),y(0))=(0,\delta_{0})$. By the concavity of $x(t)$ and $x(0)=0$, we have that $x(t) \leq \delta_{0} t \leq \delta_{0}\sigma <1$, for all $t\in\mathopen{[}0,\sigma\mathclose{]}$.

We assume by contradiction that there exists $t_{0} \in \mathopen{]}0,\sigma\mathclose{]}$ such that $y(t_{0})=0$, which is the first maximum point of the function $x(t)$. Therefore $0<x(t)\leq \delta_{0}\sigma$, for all $t\in\mathopen{]}0,t_{0}\mathclose{]}$. We define $\omega := \pi / 2t_{0}$, so that $\omega \geq \pi / 2\sigma$.
Next, we obtain
\begin{align*}
0 
&= \Bigl{[} y(t) \sin(\omega t) - \omega x(t) \cos(\omega t) \Bigr{]}_{t=0}^{t= t_{0}}\\
&= \int_{0}^{t_{0}} \biggl{[}\omega^{2} - \lambda a^{+}(\xi) \dfrac{g(x(\xi))}{x(\xi)}\biggr{]} x(\xi) \sin(\omega\xi) \,d\xi >0,
\end{align*}
a contradiction.
Then the lemma is proved.
\end{proof}

\begin{lemma}\label{lem-A1.3}
Let $\nu_{0},\nu_{1}$ be such that $0 < \nu_{1} < \nu_{0} < 1$ and $t_{1} \in\mathopen{]}0,\sigma\mathclose{[}$.
Given 
\begin{equation}\label{eq-lambda1}
\lambda^{\star}(\nu_{0},\nu_{1},t_{1}) := \dfrac{\nu_{0}-\nu_{1}}{g_{m}(\nu_{1},\nu_{0})\int_{0}^{t_{1}} A^{+}(0,\xi) \,d\xi}
\end{equation}
and $0<\gamma_{1} \leq (\nu_{0}-\nu_{1})/t_{1}$, for every $\lambda > \lambda^{\star}(\nu_{0},\nu_{1},t_{1})$, if $(x(t),y(t))$ is the solution of \eqref{IVP-1} with initial conditions $(x(0),y(0))=(\nu_{0},0)$, then $(x(t_{1}),y(t_{1}))\in\mathopen{]}-\infty,\nu_{1}\mathclose{[}\times\mathopen{]}-\infty,-\gamma_{1}\mathclose{[}$.
\end{lemma}

\begin{proof}
Let $\nu_{0},\nu_{1},t_{1},\gamma_{1}$ and $\lambda^{\star}(\nu_{0},\nu_{1},t_{1})$ be fixed as in the statement. Let $\lambda > \lambda^{\star}(\nu_{0},\nu_{1},t_{1})$ and $(x(t),y(t))$ be the solution of \eqref{IVP-1} with $(x(0),y(0))=(\nu_{0},0)$.

We suppose by contradiction that $(x(t_{1}),y(t_{1}))\notin\mathopen{]}-\infty,\nu_{1}\mathclose{[}\times\mathopen{]}-\infty,-\gamma_{1}\mathclose{[}$. Two possible situations can occur: $x(t_{1}) \geq \nu_{1}$ and $y(t_{1}) \geq -\gamma_{1}$.

First, let $x(t_{1}) \geq \nu_{1}$. The monotonicity of $x(t)$ on $\mathopen{[}0,\sigma\mathclose{]}$ ensures that $0 < \nu_{1} \leq x(t) \leq \nu_{0} < 1$, for all $t\in\mathopen{[}0,t_{1}\mathclose{]}$. Since $y'(t) \leq -\lambda a^{+}(t)g_{m}(\nu_{1},\nu_{0})$ on $\mathopen{[}0,t_{1}\mathclose{]}$, we deduce that $y(t) \leq -\lambda g_{m}(\nu_{1},\nu_{0}) A^{+}(0,t)$ on $\mathopen{[}0,t_{1}\mathclose{]}$. Consequently
\begin{equation*}
x(t) \leq x(0)-\lambda g_{m}(\nu_{1},\nu_{0}) \int_{0}^{t}A^{+}(0,\xi)\,d\xi, \quad \text{for all } \, t\in\mathopen{[}0,t_{1}\mathclose{]},
\end{equation*}
and, since $\lambda > \lambda^{\star}(\nu_{0},\nu_{1},t_{1})$, in particular we have
\begin{equation*}
x(t_{1}) \leq \nu_{0}-\lambda g_{m}(\nu_{1},\nu_{0}) \int_{0}^{t_{1}}A^{+}(0,\xi)\,d\xi < \nu_{1},
\end{equation*}
a contradiction.

Secondly, let $y(t_{1})\geq-\gamma_{1}$, and thus $y(t) \geq -\gamma_{1}$, for all $t\in\mathopen{[}0,t_{1}\mathclose{]}$. By an integration, we have
\begin{equation*}
x(t_{1}) = \nu_{0} + \int_{0}^{t_{1}} y(\xi)\,d\xi \geq \nu_{0}-\gamma_{1} t_{1} \geq \nu_{1}.
\end{equation*}
A contradiction is achieved as above and the thesis follows.
\end{proof}

\begin{lemma}\label{lem-A1.4}
Let $\lambda>0$, $\nu_{1}\in\mathopen{]}0,1\mathclose{[}$ and $t_{1}\in\mathopen{]}0,\sigma\mathclose{[}$. For every $\gamma_{1} \geq \nu_{1}/(\sigma-t_{1})$, if $(x(t),y(t))$ is a solution of \eqref{IVP-1} with $(x(t_{1}),y(t_{1}))\in\mathopen{]}-\infty,\nu_{1}\mathclose{]}\times\mathopen{]}-\infty,-\gamma_{1}\mathclose{]}$, then $(x(\sigma),y(\sigma))\in\mathopen{]}-\infty,0\mathclose{]}\times\mathopen{]}-\infty,-\gamma_{1}\mathclose{]}$.
\end{lemma}

\begin{proof}
Let $\lambda, \nu_{1}, t_{1}, \gamma_{1}$ be fixed as in the statement.
Let $(x(t),y(t))$ be a solution of \eqref{IVP-1} with $(x(t_{1}),y(t_{1}))\in\mathopen{]}-\infty,\nu_{1}\mathclose{]}\times\mathopen{]}-\infty,-\gamma_{1}\mathclose{]}$.
Since $y'(t)\leq0$ on $\mathopen{[}0,\sigma\mathclose{]}$, we immediately obtain that $y(t) \leq y(t_{1}) \leq -\gamma_{1}$, for all $t\in \mathopen{[}t_{1},\sigma\mathclose{]}$. Next, by an integration, we obtain
\begin{equation*}
x(\sigma) = x(t_{1}) + \int_{t_{1}}^{\sigma} y(\xi) \,d\xi \leq \nu_{1} - \gamma_{1}(\sigma-t_{1}) \leq 0.
\end{equation*}
Then the lemma is proved.
\end{proof}

\begin{lemma}\label{lem-A1.5}
Let $\lambda>0$, $\omega\in\mathopen{]}0,\pi/2\mathclose{[}$ and $\nu_{1}\in\mathopen{]}0,1\mathclose{[}$. Then, there exists $\hat{\varepsilon}=\hat{\varepsilon}(\lambda,\omega)>0$ such that for any $\varepsilon\in\mathopen{]}0,\hat{\varepsilon}\mathclose{[}$ there exists $\nu_{\varepsilon}\in\mathopen{]}0,\nu_{1}\mathclose{[}$ such that the following holds: for any fixed $\nu\in\mathopen{]}0,\nu_{\varepsilon}\mathclose{]}$, the solution $(x(t),y(t))$ of \eqref{IVP-1} with $(x(0),y(0))=(\nu,0)$ satisfies
$x(t)>0$ and $-\tan(\omega) x(t) \leq y(t) \leq 0$, for all $t\in \mathopen{[}0,\sigma\mathclose{]}$.
\end{lemma}

\begin{proof}
Let $\lambda, \omega, \nu_{1}$ be fixed as in the statement.
First of all, we notice that, by hypothesis $(g_{0})$, for all $\varepsilon>0$ there exists $\nu_{\varepsilon}\in\mathopen{]}0,\nu_{1}\mathclose{[}$ such that $g(s) \leq \varepsilon s$, for all $s\in\mathopen{[}0,\delta_{\varepsilon}\mathclose{]}$.
Let $\nu\in\mathopen{]}0,\delta_{\varepsilon}\mathclose{]}$ and let $(x(t),y(t))$ be the solution of \eqref{IVP-1} with $(x(0),y(0))=(\nu,0)$.
Then, we fix $\hat{\varepsilon}=\hat{\varepsilon}(\lambda,\nu)>0$ such that
\begin{equation}\label{eq-epsilon}
\sqrt{\lambda\|a^{+}\|_{\infty}\varepsilon} \tan\bigr{(}\sigma\sqrt{\lambda\|a^{+}\|_{\infty}\varepsilon} \bigr{)} < \tan(\omega), \quad \text{for all } \, \varepsilon\in\mathopen{]}0,\hat{\varepsilon}\mathclose{[}.
\end{equation}

We are going to prove that $x(t)>0$ for all $t \in\mathopen{[}0,\sigma\mathclose{]}$.
By contradiction, we assume the existence of a maximal interval $\mathopen{[}0,\hat{\sigma}\mathclose{[}$ with $\hat{\sigma}\leq\sigma$ such that $x(t)>0$ for all $t \in\mathopen{[}0,\hat{\sigma}\mathclose{[}$.
For convenience, we introduce the polar coordinates
\begin{equation*}
x(t) = \rho(t) \cos(\vartheta(t)), \quad y(t) = \rho(t) \sin(\vartheta(t)),
\end{equation*}
and so
\begin{equation*}
\vartheta(t) = \arctan \biggl{(} \dfrac{y(t)}{x(t)} \biggr{)}, \quad t\in \mathopen{[}0,\hat{\sigma}\mathclose{[}.
\end{equation*}
From $\vartheta(0)=0$ and 
\begin{equation*}
\vartheta'(t) = \dfrac{y'(t)x(t)-x'(t)y(t)}{x^{2}(t)+y^{2}(t)} =
\dfrac{-\lambda a^{+}(t) g(x(t)) x(t) - y^{2}(t)}{\rho^{2}(t)} \leq 0,
\end{equation*}
it follows that $- \pi/2<\vartheta(t) \leq 0$, for all $t\in \mathopen{[}0,\hat{\sigma}\mathclose{[}$.
Furthermore, given $\varepsilon\in\mathopen{]}0,\hat{\varepsilon}\mathclose{[}$, we have
\begin{align*}
-\vartheta'(t) 
&= \dfrac{\lambda a^{+}(t) g(x(t)) x(t) + y^{2}(t)}{\rho^{2}(t)}
\leq \dfrac{\lambda a^{+}(t) \varepsilon x^{2}(t) + y^{2}(t)}{\rho^{2}(t)}
\\ &\leq \lambda \|a^{+}\|_{\infty} \varepsilon \cos^{2}(\vartheta(t)) + \sin^{2}(\vartheta(t)),\quad \text{for all } \, t\in \mathopen{[}0,\hat{\sigma}\mathclose{[}.
\end{align*}
By an integration, for all $t\in \mathopen{[}0,\hat{\sigma}\mathclose{[}$, the following holds 
\begin{align*}
\sigma \geq \hat{\sigma} \geq t =  \int_{0}^{t} d\xi &\geq -\int_{\vartheta(0)}^{\vartheta(t)} \dfrac{d\zeta}{\lambda\|a^{+}\|_{\infty}\varepsilon \cos^{2}(\zeta) + \sin^{2}(\zeta) }
\\&= \int_{\vartheta(t)}^{0} \dfrac{d\zeta}{\cos^{2}(\zeta) \bigr{(}\lambda\|a^{+}\|_{\infty}\varepsilon + \tan^{2}(\zeta) \bigl{)} }
\\ &= -\int_{\tan(\vartheta(t))}^{0} \dfrac{dz}{\lambda\|a^{+}\|_{\infty}\varepsilon + z^{2}}
\\ &= \dfrac{1}{\sqrt{\lambda\|a^{+}\|_{\infty}\varepsilon}}\arctan\biggl{(}\dfrac{\tan|\vartheta(t)|}{\sqrt{\lambda\|a^{+}\|_{\infty}\varepsilon}} \biggr{)}.
\end{align*}
Then, for all $t\in \mathopen{[}0,\hat{\sigma}\mathclose{[}$ we obtain
\begin{equation*}
|\vartheta(t)| \leq \arctan\Bigl{(}\sqrt{\lambda\|a^{+}\|_{\infty}\varepsilon} \tan\bigr{(}\sigma\sqrt{\lambda\|a^{+}\|_{\infty}\varepsilon} \bigr{)}\Bigr{)}
\end{equation*}
and thus $-\nu<\vartheta(t)\leq 0$.
Consequently, the continuity of $\vartheta(t)$ implies $\vartheta(\hat{\sigma})\geq-\nu>-\pi/2$, and so $x(\hat{\sigma})>0$. This contradicts the definition of $\hat{\sigma}$.
Accordingly, $x(t)>0$ for all $t\in\mathopen{[}0,\sigma\mathclose{]}$ and so the thesis follows from the above discussion.
\end{proof}

\subsection{Analysis in $\mathopen{[}\tau,T\mathclose{]}$}\label{section-A.2}

The analysis of system \eqref{IVP-1} in the interval $\mathopen{[}\tau,T\mathclose{]}$ can be performed in analogy with the results given in Section~\ref{section-A.1} due to the particular structure of the problem.
Accordingly, the following lemmas state that, for $\lambda>0$, the solution $(x(t),y(t))$ of \eqref{IVP-1} with initial conditions $(x(T),y(T))=(x_{T},y_{T})$ satisfies
\begin{itemize}
\item $(x(\tau),y(\tau))\in\mathopen{]}1,+\infty\mathclose{[}\times\mathopen{]}-\infty,0\mathclose{[}$, taking $(x_{T},y_{T})=(0,\delta)$ with $\delta>0$ sufficiently large (cf.~Lemma~\ref{lem-A2.2});

\item $(x(t),y(t))\in\mathopen{]}0,1\mathclose{[}\times\mathopen{]}-\infty,0\mathclose{[}$, for all $t\in\mathopen{[}\tau,T{[}$, taking $(x_{T},y_{T})=(0,-\delta)$ with $\delta>0$ sufficiently small (cf.~Lemma~\ref{lem-A2.1});

\item $(x(\tau),y(\tau))\in\mathopen{]}-\infty,0\mathclose{]}\times\mathopen{]}0,+\infty\mathclose{[}$, taking $(x_{T},y_{T})=(\nu,0)$ with $\nu\in\mathopen{]}0,1\mathclose{[}$ and $\lambda$ sufficiently large (cf.~Lemma~\ref{lem-A2.3} and Lemma~\ref{lem-A2.4});

\item $(x(t),y(t))$ belongs to a small angular region contained in $\mathopen{]}0,1\mathclose{[}\times\mathopen{]}0,+\infty\mathclose{[}$, for all $t\in\mathopen{[}\tau,T\mathclose{[}$, taking $(x_{T},y_{T})=(\nu,0)$ with $\nu>0$ sufficiently small (cf.~Lemma~\ref{lem-A2.5}).
\end{itemize}
The corresponding proofs are omitted since they can be adapted straightforward from the ones contained in Section~\ref{section-A.1}.

\begin{lemma}\label{lem-A2.2}
Let $\lambda>0$. There exists $\tilde{\delta}>0$ such that the solution $(x(t),y(t))$ of \eqref{IVP-1} with $(x(T),y(T))=(0,-\tilde{\delta})$ satisfies $(x(\tau),y(\tau))\in\mathopen{]}1,+\infty\mathclose{[}\times\mathopen{]}-\infty,0\mathclose{[}$.
\end{lemma}

\begin{lemma}\label{lem-A2.1}
Let $\lambda>0$ and let $\delta_{0}\in\mathopen{]}0,1/(T-\tau)\mathclose{[}$ be such that
\begin{equation*}
\biggl{(}\dfrac{\pi}{2(T-\tau)}\biggr{)}^{2} - \lambda a^{+}(t) \sup_{s\in\mathopen{]}0,\delta_{0}(T-\tau)\mathclose{]}}\dfrac{g(s)}{s} > 0, \quad \text{for a.e.~$t\in\mathopen{[}\tau,T\mathclose{]}$}.
\end{equation*}
If $(x(t),y(t))$ is the solution of \eqref{IVP-1} with initial conditions $(x(T),y(T))=(0,-\delta_{0})$, then $(x(t),y(t))\in\mathopen{]}0,1\mathclose{[}\times\mathopen{]}-\infty,0\mathclose{[}$, for all $t\in\mathopen{[}\tau,T{[}$.
\end{lemma}

\begin{lemma}\label{lem-A2.3}
Let $\nu_{1},\nu_{T}$ be such that $0 < \nu_{1} < \nu_{T} < 1$ and $t_{1}\in\mathopen{]}\tau,T\mathclose{[}$. 
Given 
\begin{equation}\label{eq-lambda2}
\lambda^{\star\star}(\nu_{1},\nu_{T},t_{1}) := \dfrac{\nu_{T}-\nu_{1}}{g_{m}(\nu_{1},\nu_{T})\int_{t_{1}}^{T} A^{+}(\xi,T) \,d\xi}
\end{equation}
and $0<\gamma_{1}\leq(\nu_{T}-\nu_{1})/(T-t_{1})$, for every $\lambda > \lambda^{\star\star}(\nu_{1},\nu_{T},t_{1})$, if $(x(t),y(t))$ is the solution of \eqref{IVP-1} with initial conditions $(x(T),y(T))=(\nu_{T},0)$, then $(x(t_{1}),y(t_{1}))\in\mathopen{]}-\infty,\nu_{1}\mathclose{[}\times\mathopen{]}\gamma_{1},+\infty\mathclose{[}$.
\end{lemma}

\begin{lemma}\label{lem-A2.4}
Let $\lambda>0$, $\nu_{1}\in\mathopen{]}0,1\mathclose{[}$ and $t_{1}\in\mathopen{]}\tau,T\mathclose{[}$. 
For every $\gamma_{1} \geq \nu_{1}/(t_{1}-\tau)$, if $(x(t),y(t))$ is a solution of \eqref{IVP-1} with $(x(t_{1}),y(t_{1}))\in\mathopen{]}-\infty,\nu_{1}\mathclose{]}\times\mathopen{[}\gamma_{1},+\infty\mathclose{]}$, then $(x(\tau),y(\tau))\in\mathopen{]}-\infty,0\mathclose{]}\times\mathopen{[}\gamma_{1},+\infty\mathclose{[}$.
\end{lemma}

\begin{lemma}\label{lem-A2.5}
Let $\lambda>0$, $\omega\in\mathopen{]}0,\pi/2\mathclose{[}$ and $\nu_{1}\in\mathopen{]}0,1\mathclose{[}$. Then, there exists $\hat{\varepsilon}=\hat{\varepsilon}(\lambda,\omega)>0$ such that for any $\varepsilon\in\mathopen{]}0,\hat{\varepsilon}\mathclose{[}$ there exists $\nu_{\varepsilon}\in\mathopen{]}0,\nu_{1}\mathclose{[}$ such that the following holds: for any fixed $\nu\in\mathopen{]}0,\nu_{\varepsilon}\mathclose{]}$, the solution $(x(t),y(t))$ of \eqref{IVP-1} with $(x(T),y(T))=(\nu,0)$ satisfies
$x(t)>0$ and $0 \leq y(t) \leq \tan(\omega) \, x(t)$ for all $t\in \mathopen{[}\tau,T\mathclose{]}$.
\end{lemma}

\subsection{Analysis in $\mathopen{[}\sigma,\tau\mathclose{]}$}\label{section-A.3}

We deal with system $(\mathcal{S}_{\lambda,\mu})$ in the interval $\mathopen{[}\sigma,\tau\mathclose{]}$, where it is written as
\begin{equation}\label{IVP-2}
\begin{cases}
\, x' = y, \\
\, y' = \mu a^{-}(t) g(x).
\end{cases}
\end{equation}
For $\mu>0$, we focus our attention on the solutions to Cauchy problems associated with \eqref{IVP-2} at a given point $\kappa\in\mathopen{]}\sigma,\tau\mathclose{[}$.

The following lemmas describe the dynamics on the two subintervals $\mathopen{[}\sigma,\kappa\mathclose{]}$ and $\mathopen{[}\kappa,\tau\mathclose{]}$. More precisely, we prove the following.
\begin{itemize}
\item The solution $(x(t),y(t))$ of \eqref{IVP-2} with initial conditions $(x(\sigma),y(\sigma))\in\mathopen{]}0,1\mathclose{[}\times\mathbb{R}$ satisfies $(x(\kappa),y(\kappa))\in\mathopen{[}1,+\infty\mathclose{[}\times\mathopen{]}0,+\infty\mathclose{[}$ for $\mu$ sufficiently large (cf.~Lemma~\ref{lem-A3.1} and Lemma~\ref{lem-A3.2}).

\item The solution $(x(t),y(t))$ of \eqref{IVP-2} with initial conditions $(x(\tau),y(\tau))\in\mathopen{]}0,1\mathclose{[}\times\mathbb{R}$ satisfies $(x(\kappa),y(\kappa))\in\mathopen{[}1,+\infty\mathclose{[}\times\mathopen{]}-\infty,0\mathclose{[}$ for $\mu$ sufficiently large (cf.~Lemma~\ref{lem-A3.3} and Lemma~\ref{lem-A3.4}).
\end{itemize}

\begin{lemma}\label{lem-A3.1}
Let $\nu_{\sigma},\nu_{2}$ be such that $0 < \nu_{\sigma} < \nu_{2} < 1$ and $\omega_{\sigma} > 0$.
Given 
\begin{equation*}
\sigma< t_{2} \leq \min\biggl{\{}\sigma + \dfrac{\nu_{\sigma}}{2\omega_{\sigma}},\kappa \biggr{\}}, \quad 0 < \omega \leq \dfrac{\nu_{2}-\nu_{\sigma}}{t_{2}-\sigma},
\end{equation*} 
and
\begin{equation}\label{eq-mu}
\mu^{\star}(\nu_{2},\nu_{\sigma},t_{2},\omega_{\sigma}) := \dfrac{\nu_{2} - \nu_{\sigma} + (t_{2}-\sigma) \omega_{\sigma}}{g_{m}(\nu_{\sigma}/2,\nu_{2}) \int_{\sigma}^{t_{2}} A^{-}(\sigma,\xi) \,d\xi},
\end{equation}
for every $\mu>\mu^{\star}(\nu_{2},\nu_{\sigma},t_{2},\omega_{\sigma})$, if $(x(t),y(t))$ is a solution of \eqref{IVP-2} with $(x(\sigma),y(\sigma)) \in \{\nu_{\sigma}\}\times\mathopen{[}-\omega_{\sigma},+\infty\mathclose{[}$, then $(x(t_{2}),y(t_{2})) \in \mathopen{]}\nu_{2},+\infty\mathclose{[} \times \mathopen{]}\omega,+\infty\mathclose{[}$.
\end{lemma}

\begin{proof}
Let $\nu_{\sigma},\nu_{2},\omega_{\sigma},t_{2},\omega$ and $\mu^{\star}(\nu_{2},\nu_{\sigma},t_{2},\omega_{\sigma})$ be fixed as in the statement.
For $\mu>\mu^{\star}(\nu_{2},\nu_{\sigma},t_{2},\omega_{\sigma})$, let $(x(t),y(t))$ be a solution of \eqref{IVP-2} with $(x(\sigma),y(\sigma)) \in \{\nu_{\sigma}\}\times\mathopen{[}-\omega_{\sigma},+\infty\mathclose{[}$.

We suppose by contradiction that $(x(t_{2}),y(t_{2})) \notin \mathopen{]}\nu_{2},+\infty\mathclose{[} \times \mathopen{]}\omega,+\infty\mathclose{[}$. Two possible situations can occur: $x(t_{2}) \leq \nu_{2}$ and $y(t_{2}) \leq \omega$.

First, let $x(t_{2})\leq\nu_{2}$. Recalling that the function $x(t)$ is convex in $\mathopen{[}\sigma,\kappa\mathclose{]}$ and that $\nu_{2}>\nu_{\sigma}$, we obtain that 
$x(t) \leq \nu_{2}$, for all $t\in \mathopen{[}\sigma,t_{2}\mathclose{]}$.
We notice that $y'(t)\geq0$ on $\mathopen{[}\sigma,\kappa\mathclose{]}$ and $y(\sigma) \geq -\omega_{\sigma}$, hence
\begin{equation*}
x(t) \geq -\omega_{\sigma} t + \nu_{\sigma} + \omega_{\sigma} \sigma,
\quad \text{for all } \, t\in\mathopen{[}\sigma,\kappa\mathclose{]}.
\end{equation*}
By the choice of $t_{2}$, the previous inequality yields 
\begin{equation*}
x(t) \geq \dfrac{\nu_{\sigma}}{2}, \quad \text{for all } \, t\in \mathopen{[}\sigma,t_{2}\mathclose{]}.
\end{equation*}
By integrating twice $y'= \mu a^{+}(t)g(x)$, for every $t\in \mathopen{[}\sigma,t_{2}\mathclose{]}$, we have
\begin{equation*}
x(t) = x(\sigma) + \int_{\sigma}^{t} y(\xi) \,d\xi 
= \nu_{\sigma} + (t-\sigma) y(\sigma) + \mu \int_{\sigma}^{t} \int_{\sigma}^{z} a^{-}(\xi) g(x(\xi)) \,d\xi dz.
\end{equation*}
Then, it follows that
\begin{align*}
\nu_{2} \geq x(t_{2}) \geq \nu_{\sigma} - (t_{2}-\sigma) \omega_{\sigma} + \mu g_{m}(\nu_{\sigma}/2,\nu_{2}) \int_{\sigma}^{t_{2}} A^{-}(\sigma,\xi) \,d\xi > \nu_{2},
\end{align*}
a contradiction.

Secondly, let $y(t_{2}) \leq \omega$. Then, $y(t) \leq \omega$ for all $t\in \mathopen{[}\sigma,t_{2}\mathclose{]}$, thus $x(t_{2}) \leq \nu_{\sigma} + \omega (t_{2}-\sigma) \leq \nu_{2}$. A contradiction is obtained as above and the thesis follows.
\end{proof}

\begin{lemma}\label{lem-A3.2}
Let $\mu>0$, $\nu_{2}\in\mathopen{]}0,1\mathclose{[}$ and $t_{2}\in\mathopen{]}\sigma,\kappa\mathclose{[}$. For every $\omega\geq(1-\nu_{2})/(\kappa-t_{2})$, if $(x(t),y(t))$ is a solution of \eqref{IVP-2} with $(x(t_{2}),y(t_{2}))\in\mathopen{[}\nu_{2},+\infty\mathclose{[}\times\mathopen{[}\omega,+\infty\mathclose{[}$, then 
$(x(\kappa),y(\kappa))\in\mathopen{[}1,+\infty\mathclose{[}\times\mathopen{[}\omega,+\infty\mathclose{[}$.
\end{lemma}

\begin{proof}
Let $\mu,\nu_{2},t_{2},\omega$ be fixed as in the statement.
Let $(x(t),y(t))$ be a solution of \eqref{IVP-2} with $(x(t_{2}),y(t_{2}))\in\mathopen{[}\nu_{2},+\infty\mathclose{[}\times\mathopen{[}\omega,+\infty\mathclose{[}$.
From $y'(t)\geq 0$ on $\mathopen{[}\sigma,\kappa\mathclose{]}$, we deduce that $y(t)\geq y(t_{2}) \geq \omega$ for every $t\in \mathopen{[}t_{2},\kappa\mathclose{]}$. In particular, $y(\kappa)\geq \omega$. Moreover, we have 
\begin{equation*}
x(\kappa) = x(t_{2}) + \int_{t_{2}}^{\kappa} y(\xi) \,d\xi \geq \nu_{2} + \omega (\kappa - t_{2}) \geq 1.
\end{equation*}
The thesis follows.
\end{proof}

\begin{lemma}\label{lem-A3.3}
Let $\nu_{\tau},\nu_{2}$ be such that $0 < \nu_{\tau} < \nu_{2} < 1$ and $\omega_{\tau} > 0$.
Given 
\begin{equation*}
\max\biggl{\{}\tau - \dfrac{\nu_{\tau}}{2\omega_{\tau}},\kappa \biggr{\}} \leq t_{2} < \tau, \quad 0 < \omega \leq \dfrac{\nu_{2}-\nu_{\tau}}{\tau-t_{2}},
\end{equation*} 
and
\begin{equation*}
\mu^{\star}(\nu_{2},\nu_{\tau},t_{2},\omega_{\tau}) := \dfrac{\nu_{2} - \nu_{\tau} + (\tau - t_{2}) \omega_{\tau}}{g_{m}(\nu_{\tau}/2,\nu_{2}) \int_{t_{2}}^{\tau} A^{-}(\xi,\tau) \,d\xi},
\end{equation*}
for every $\mu>\mu^{\star}(\nu_{2},\nu_{\tau},t_{2},\omega_{\tau})$, if $(x(t),y(t))$ is a solution of \eqref{IVP-2} with $(x(\tau),y(\tau)) \in \{\nu_{\tau}\}\times\mathopen{]}-\infty,\omega_{\tau}\mathclose{[}$, then $(x(t_{2}),y(t_{2})) \in \mathopen{]}\nu_{2},+\infty\mathclose{[} \times \mathopen{]}-\infty,-\omega\mathclose{[}$.
\end{lemma}

\begin{lemma}\label{lem-A3.4}
Let $\mu>0$, $\nu_{2}\in\mathopen{]}0,1\mathclose{[}$ and $t_{2}\in\mathopen{]}\kappa,\tau\mathclose{[}$. For every $\omega\geq(1-\nu_{2})/(t_{2}-\kappa)$, if $(x(t),y(t))$ is a solution of \eqref{IVP-2} with $(x(t_{2}),y(t_{2}))\in\mathopen{[}\nu_{2},+\infty\mathclose{[}\times\mathopen{]}-\infty,-\omega\mathclose{]}$, then 
$(x(\kappa),y(\kappa))\in\mathopen{[}1,+\infty\mathclose{[}\times\mathopen{]}-\infty,-\gamma\mathclose{]}$.
\end{lemma}

The proofs of Lemma~\ref{lem-A3.3} and Lemma~\ref{lem-A3.4} are omitted since analogous to the ones of Lemma~\ref{lem-A3.1} and Lemma~\ref{lem-A3.2}.

\section*{Acknowledgments}

The authors are grateful to Professor Yuan Lou for inspiring to pursue investigations on the subject of the paper and to Professors Dimitri Breda and Fabio Zanolin for interesting discussions.

\bibliographystyle{elsart-num-sort}
\bibliography{FeSo_biblio}

\bigskip
\begin{flushleft}

{\small{\it Preprint}}

{\small{\it December 2017}}

\end{flushleft}

\end{document}